\documentclass[reqno,final]{amsart}

\usepackage{showlabels}
\usepackage[utf8]{inputenc} 
\usepackage[foot]{amsaddr} 
\usepackage{amssymb}
\usepackage{verbatim} 
\usepackage{amssymb}
\usepackage{amsthm}
\usepackage{amsmath}
\usepackage{amsfonts}
\usepackage{latexsym}
\usepackage{mathtools} 
\usepackage{enumitem} 
\usepackage[english]{babel}
\usepackage{esint}
\usepackage{cite}
\usepackage{calrsfs}
\usepackage{hyperref} \usepackage{color}

\theoremstyle{plain}

\newtheorem{theorem}{Theorem}
\newtheorem{lemma}[theorem]{Lemma}

\newtheorem{proposition}[theorem]{Proposition}
\newtheorem{remark}[theorem]{Remark}

\newtheorem{corollary}[theorem]{Corollary}

\numberwithin{equation}{section}
\numberwithin{theorem}{section}

\newcommand{\eqdef }{\overset{\mbox{\tiny{def}}}{=}}
\newcommand{\rth}{{\mathbb{R}^3}}

\usepackage{todonotes}
\newcommand{\pv}{p}
\newcommand{\pZ}{\pv^0}
\newcommand{\pZp }{\pv'^{0}}
\newcommand{\qv}{q}
\newcommand{\qZ }{\qv^0}
\newcommand{\qZp }{\qv'^{0}}

\title[Steady Solutions to the Relativistic Boltzmann Equation in a Slab]{Steady Solutions to the Relativistic Boltzmann Equation in a Slab}

\author[J. W. Jang]{Jin Woo Jang$^\dagger$}
\address{$^\dagger$ Department of mathematics, POSTECH (Pohang University of Science and Technology), Pohang 37673, Republic of Korea. \href{mailto:jangjw@postech.ac.kr}{jangjw@postech.ac.kr} }

\author[S.-B. Yun]{Seok-Bae Yun$^\ddagger$}
\address{$^\ddagger$Department of mathematics, Sungkyunkwan University, Suwon 440-746, Republic of Korea. \href{mailto:sbyun01@skku.edu}{sbyun01@skku.edu} }

\AtBeginDocument{%
   \def\MR#1{}
}

\begin{document}

\date{\today}

\let\thefootnote\relax
\addtocounter{footnote}{-1}\let\thefootnote\svthefootnote
\begin{abstract}
We study steady solutions to the relativistic Boltzmann equation with hard-sphere
interactions in a slab geometry.
Under a spatial symmetry assumption in the transverse variables $x_2$ and $x_3$,
the problem reduces to a one-dimensional spatial slab $x_1 \in [0,1]$ while
retaining full three-dimensional momentum dependence.
For non-negative inflow boundary conditions prescribed at $x_1=0$ and $x_1=1$,
we prove the existence and uniqueness of a stationary solution in a weighted
$L^1_p L^\infty_{x_1}$ framework, together with exponential decay in momentum.

Our analysis treats the full slab domain and does not rely on any smallness
assumption on the slab width.
We establish sharp coercivity and continuity estimates for the collision
frequency, together with weighted convolution and pointwise bounds for the
nonlinear gain term.
These estimates generate and propagate a $(-\Delta_p)^{-1}$-type regularity
within the solution framework, which plays a crucial role in the existence and
uniqueness argument.
In addition, we obtain uniform weighted integrability of the solution over
arbitrary two-dimensional hyperplanes through the origin.
This hyperplane estimate is derived as a genuinely \textit{a posteriori}
regularity property, without imposing any \textit{a priori} hyperplane bounds, and
follows from a Lorentz-invariant geometric reduction.
\end{abstract}

\setcounter{tocdepth}{1}

\allowdisplaybreaks

\maketitle
\tableofcontents

\thispagestyle{empty}
\section{Introduction}

This work concerns the stationary relativistic Boltzmann equation.
Incorporating Einstein's theory of special relativity, the relativistic
Boltzmann equation describes the statistical evolution of a dilute gas
whose particles move at relativistic speeds
\cite{DeGroot,MR1898707}.
It provides one of the fundamental dynamical models in relativistic kinetic
theory and plays an important role in the mathematical description of
high–energy particle systems, astrophysical plasmas, and relativistic gas
flows.

While the time–dependent relativistic Boltzmann equation has been studied
extensively in recent decades, much less is known about its stationary
counterpart.  Stationary kinetic equations arise naturally when a gas
reaches a time–independent configuration under the influence of boundary
forcing or external constraints.  Such steady states describe long–time
regimes of nonequilibrium systems, boundary–driven flows, and kinetic
transport in confined geometries.  Understanding their existence and
structure is therefore a fundamental step toward the mathematical theory
of kinetic boundary value problems.

\subsection{Stationary solutions to collisional kinetic equations}

As emphasized by Cercignani~\cite{MR1744523}, establishing the existence of
stationary solutions to the Boltzmann equation under physically meaningful
boundary conditions is a central problem in kinetic theory.
In contrast with the Cauchy problem, however, the theory of stationary
boundary-value problems remains considerably less developed
\cite{Vil02}.  
One major reason is that stationary equations lack the smoothing and
dissipative mechanisms that are often available in time-dependent
problems.  As a consequence, many of the compactness and energy methods
developed for the evolution equation do not directly apply in the steady
setting.

The study of stationary solutions to collisional kinetic equations,
and in particular to the Boltzmann equation, nevertheless has a long
history.  Existence and regularity results for stationary boundary-value
problems have been obtained in a variety of settings; see, for example,
\cite{MR459450,MR1413668,MR2134697,MR1991144,MR1375351,MR1842024,
MR2130642,MR2122556,MR1766902,MR4419612,EGKM-18-AP,MR3085665,
MR4131014,MR406275,MR3926521,2503.12587} and the references therein.
These works illustrate both the richness of stationary kinetic phenomena
and the substantial analytical difficulties that arise when one attempts
to construct steady solutions under realistic boundary conditions.

For relativistic kinetic equations, the study of stationary states
dates back to the introduction of the Maxwell--J\"uttner distribution by
J\"uttner~\cite{Juttner} in 1911.
Its mathematical properties were later revisited in several works,
including the study of Abonyi~\cite{MR122493} in 1960.
More recently, Hwang and Yun~\cite{1801.08382} investigated stationary
solutions to the relativistic BGK model of Marle type in a slab geometry.
Very recently, Wang, Li, and Jiang~\cite{2411.06533} studied the
stationary relativistic Boltzmann equation on a half-line for hard
potentials and proved that boundary profiles given as small perturbations
of a relativistic Maxwellian decay asymptotically to equilibrium in the
$L^\infty$ sense as $x_1 \to \infty$.
On the other hand, Ouyang and Xiao~\cite{MR4891827} established the
existence and spatial decay of boundary layer solutions to steady
relativistic Boltzmann equations with hard potentials under diffusive
reflection boundary conditions.

Taken together, these works highlight both the analytical challenges and
the relative scarcity of rigorous results for stationary relativistic
Boltzmann equations, particularly in bounded or semi-infinite spatial
domains. This motivates the present study of stationary solutions to the
relativistic Boltzmann equation in a slab geometry with inflow
boundary conditions.

\subsection{Stationary relativistic Boltzmann equation.}
The stationary relativistic Boltzmann equation that we consider in this paper reads as
\begin{align}\label{RBE}
\begin{split}
\hat{p}\cdot \nabla_x F= Q(F,F),\end{split}
\end{align}
where the particle distribution function $F(x,p)$ represents the probabilistic density function of particles at position $x=(x_1,x_2,x_3)\in \Omega \subset \rth$ with momentum $p\in\mathbb{R}^3$. 
The collision operator $Q(f,h)$ then can be written as
\begin{equation*}\label{Q original}
Q(f,h)=\frac{1}{\pZ}\int_{\mathbb{R}^3}\frac{dq}{{\qZ }}\int_{\mathbb{R}^3}\frac{dq'\hspace{1mm}}{{\qZp }}\int_{\mathbb{R}^3}\frac{dp'\hspace{1mm}}{{\pZp }}  W(p,q|p',q')(f(p')h(q')-f(p)h(q))
\end{equation*}
where we denote the particle energy and the velocity as $cp^0$ and $\hat{p}$, respectively, of a particle with the rest mass $m$ and the momentum $p$ where
\begin{equation}
    \label{pZ}
    p^0\eqdef \sqrt{m^2c^2+|p|^2},\ \textup{ and }\hat{p}\eqdef \frac{cp}{p^0}.
\end{equation}
Here the transition rate $W(p,q|p',q')$ is given by
\begin{align}\label{W}
W(p,q|p',q')=\frac{c}{2}s\sigma(g,\theta)\delta^{(4)}(p^\mu +q^\mu -p'^\mu -q'^\mu),
\end{align}
where $\sigma(g,\theta)$ is the scattering kernel measuring the interactions between particles, $s$ and $g$ are two variables depending on $(p^0,p)$ and $(q^0,q)$ (whose definition will further be introduced in \eqref{s} and \eqref{g}), and the Dirac-delta function, $\delta^{(4)}$, enforces the conservation of energy and momentum \eqref{collision.invariants}.  Without loss of generality, we normalize the speed of light $c$ and the rest mass $m$ to be equal to 1 through the rest of the paper.   Other necessary notations will be defined in the next section. 
\subsection{Hypothesis on the scattering kernel}\label{sec.ang.hypo}
Throughout the paper, we assume that the relativistic Boltzmann cross-section $\sigma$ will be written as the following product-type form:
\begin{equation}
    \label{kernel condition}
    \sigma(g,\theta)\approx g\sigma_0(\theta).
\end{equation} Such a representation describes the relativistic hard-sphere case. Here we assume that $\sigma_0\in L^1(\mathbb{S}^2)$ and $0\le \sigma_0(\theta)\le M$ for some positive constant $M>0.$ A generic choice of $\sigma_0(\theta)$ is $\sin^\gamma(\theta)$ for $\gamma\ge 0$, as in \cite{MR1211782} for instance. This scattering kernel covers a broad range of generic relativistic hard-sphere interactions classified by Dudy\`nski and Ekiel-Je$\dot{\textup{z}}$ewska in  \cite{Dudynski2} with angular cutoff. Without loss of generality, we normalize and assume that $M=\frac{1}{4\pi}$ such that \begin{equation}
    \label{angular assumption}\int_{\mathbb{S}^2}\sigma_0(\theta)d\omega=c_0\le 1
\end{equation} through the rest of the paper.
\subsection{Steady states under symmetry in $x_2$ and $x_3$ directions}

Throughout this work, we consider a stationary distribution function $F = F(x, p)$ that is assumed to be symmetric with respect to the transverse spatial variables $x_2$ and $x_3$. Under this symmetry, the distribution depends only on the longitudinal coordinate $x_1 \in [0,1]$ and momentum $p \in \mathbb{R}^3$, so that we may write $F(x, p) = f(x_1, p)$.

The full spatial domain is formally given as $\Omega = [0,1] \times \Omega_{x_2,x_3}$, where $\Omega_{x_2,x_3}$ is a bounded subset of $\mathbb{R}^2$. However, due to the imposed symmetry in the $x_2$ and $x_3$ directions, the problem effectively reduces to a one-dimensional slab in $x_1$. Note that if $\Omega_{x_2,x_3}$ were unbounded, the total mass of the system would be infinite unless the density is identically zero, which we exclude.

This steady-state configuration is reminiscent of the classical Couette flow in kinetic theory. We thus focus on a boundary value problem in the phase space $(x_1, p) \in [0,1] \times \mathbb{R}^3$, subject to the following inflow boundary conditions at $x_1 = 0$ and $x_1 = 1$:
\begin{equation}\label{boundary}
\begin{cases}
f(0, p) = f_L(p), &\text{for } p_1 > 0, \\
f(1, p) = f_R(p), &\text{for } p_1 < 0,
\end{cases}
\end{equation}
where $f_L$ and $f_R$ are prescribed non-negative functions describing incoming particle distributions at the boundaries. For convenience, we define the combined inflow profile
$$
f_{LR}(p) := f_L(p)\, \mathbf{1}_{p_1 > 0} + f_R(p)\, \mathbf{1}_{p_1 < 0}.
$$
\subsubsection{Boundary conditions}
\label{sec.boundary}

We impose the following assumptions on the inflow boundary profile $f_{LR}$. We first let $0<r_1<r_2<\infty$ and define
\[
Q_{r_1,r_2}\eqdef \{q\in\mathbb{R}^3:\ q_1\ge r_1 \text{ and } r_1\le |q|\le r_2\},
\] and its reflection \[
\bar{Q}_{r_1,r_2}\eqdef \{q\in\mathbb{R}^3:\ q_1\le - r_1 \text{ and } r_1\le |q|\le r_2\}.
\]Assume that $f_L\in L^1(Q_{r_1,r_2})$ and $f_R\in L^1 (\bar{Q}_{r_1,r_2})$. In addition, suppose that there exist a constant $M>0$ such that
\begin{multline}
\label{boundary condition}
f_{LR}\ge 0,  \quad
\int_{Q_{r_1,r_2}} f_L(q)\,dq>0,\quad \int_{\bar{Q}_{r_1,r_2}} f_R(q)\,dq>0,\\
\|f_{LR}\|  \le M, \quad
\|f_{LR}\|_{-1} \le M .
\end{multline}
Here the norms $\|\cdot\|$ and $\|\cdot\|_{-1}$ will be defined in
Section~\ref{sec.norm}.
These assumptions and an additional preliminary lemma (Lemma \ref{lemma m}) will ensure a uniform coercivity of the collision frequency near
the boundary and play a crucial role in the construction of steady solutions.

Formally, the relativistic Boltzmann collision operator $Q(f,f)$ satisfies the
conservation identities
\[
\int_{\mathbb{R}^3} Q(f,f)\,dp
=
\int_{\mathbb{R}^3} p^i Q(f,f)\,dp
=
\int_{\mathbb{R}^3} p^0 Q(f,f)\,dp
=0,
\qquad i=1,2,3.
\]
These identities correspond to the conservation of mass, momentum, and energy,
and lead to the following compatibility conditions on the inflow boundary data
prescribed at $x_1=0$ and $x_1=1$:
\begin{equation}
\label{conlaw}
\int_{\mathbb{R}^3}
\begin{pmatrix}
1 \\[1mm]
p \\[1mm]
p^0
\end{pmatrix}
\hat{p}_1\, f_L(p)\,dp
=
\int_{\mathbb{R}^3}
\begin{pmatrix}
1 \\[1mm]
p \\[1mm]
p^0
\end{pmatrix}
\hat{p}_1\, f_R(p)\,dp .
\end{equation}

\subsection{Solutions in the mild formulation}

We begin by recalling that the relativistic Boltzmann equation \eqref{RBE}
can be rewritten in the form
\begin{equation}
\label{recall}
\partial_{x_1} f + \frac{1}{\hat{p}_1} f \mathcal{L}f
= \frac{1}{\hat{p}_1} Q^+(f,f),
\end{equation}
where the linearized collision operator $\mathcal{L}$ is defined by
\begin{align*}
\mathcal{L}f
:= \frac{1}{p^0}
\int_{\mathbb{R}^3} \frac{dq}{q^0}
\int_{\mathbb{R}^3} \frac{dq'}{q'^0}
\int_{\mathbb{R}^3} \frac{dp'}{p'^0}
\, W(p,q \mid p',q')\, f(x_1,q),
\end{align*}
so that the loss term satisfies $Q^-(f,f)= f \mathcal{L}f$.

For each fixed $p\in\mathbb{R}^3$, equation \eqref{recall} is an ordinary
differential equation with respect to the spatial variable $x_1$.
We therefore introduce the integrating factor
\[
\exp\!\left(
\frac{1}{\hat{p}_1}
\int_{x_{1,0}}^{x_1} \mathcal{L}f(z,p)\,dz
\right),
\]
where $x_{1,0}=0$ if $p_1>0$ and $x_{1,0}=1$ if $p_1<0$.
Using the inflow boundary conditions \eqref{boundary}, we obtain the following
mild formulation of \eqref{recall}:
\begin{equation}
\label{f representation for p1}
\begin{split}
f(x_1,p)
&= \mathbf{1}_{\{p_1>0\}}
\bigg[
f_L(p)\exp\!\left(
-\frac{1}{\hat{p}_1}\int_0^{x_1} \mathcal{L}f(z,p)\,dz
\right) \\
&\quad+ \frac{1}{\hat{p}_1}
\int_0^{x_1}
\exp\!\left(
-\frac{1}{\hat{p}_1}\int_z^{x_1} \mathcal{L}f(z',p)\,dz'
\right)
Q^+(f,f)(z,p)\,dz
\bigg] \\
&\quad+ \mathbf{1}_{\{p_1<0\}}
\bigg[
f_R(p)\exp\!\left(
\frac{1}{\hat{p}_1}\int_{x_1}^{1} \mathcal{L}f(z,p)\,dz
\right) \\
&\quad- \frac{1}{\hat{p}_1}
\int_{x_1}^{1}
\exp\!\left(
\frac{1}{\hat{p}_1}\int_{x_1}^{z} \mathcal{L}f(z',p)\,dz'
\right)
Q^+(f,f)(z,p)\,dz
\bigg] \\
&=: (Af)(x_1,p).
\end{split}
\end{equation}
Here $A$ denotes the solution operator acting on $f\in\mathcal{A}$, where
the solution space $\mathcal{A}$ will be specified in
\eqref{sol space}. This formulation will serve as the starting point for the fixed-point
construction of steady solutions.

%
%
%
%
%
%

\subsection{Notations}

In this section, we introduce the notational conventions for relativistic
four-vectors and the function spaces used throughout the paper.

\begin{itemize}

\item We denote a relativistic four-vector by $p^\mu$, where
$\mu=0,1,2,3$, with components
\[
p^\mu \in \{p^0,p^1,p^2,p^3\}.
\]
Henceforth, four-vectors will simply be referred to as vectors.

\item Latin indices $a,b,i,j,k$ take values in $\{1,2,3\}$, while Greek
indices $\kappa,\lambda,\mu,\nu$ take values in $\{0,1,2,3\}$.
Indices are raised and lowered using the Minkowski metric
$\eta_{\mu\nu}$ and its inverse $(\eta^{-1})^{\mu\nu}$, defined by
\[
\eta_{\mu\nu}=(\eta^{-1})^{\mu\nu}=\mathrm{diag}(-1,1,1,1),
\]
so that $p_\mu=\eta_{\mu\nu}p^\nu$.

\item We adopt Einstein’s summation convention, whereby repeated indices with
one upper and one lower index are summed over.

\item The Lorentz inner product of two four-vectors is given by
\begin{equation}
\label{lorentz.inner.prd}
p^\mu q_\mu = p^\mu \eta_{\mu\nu} q^\nu
= -p^0 q^0 + \sum_{i=1}^3 p^i q^i .
\end{equation}

\item A four-vector $p^\mu$ satisfying the mass-shell condition
$p^\mu p_\mu=-1$ with $p^0>0$ is called an \emph{energy--momentum vector}.
In this case, $p^\mu=(p^0,p)$ with $p\in\mathbb{R}^3$, and
\[
p^0=\sqrt{1+|p|^2}.
\]
Throughout this paper, the vectors $p^\mu,q^\mu,p^{\prime\mu},q^{\prime\mu}$
appearing in the relativistic Boltzmann equation \eqref{RBE} are always
energy--momentum vectors.

\item A four-vector $a^\mu$ is called \emph{space-like} if $a^\mu a_\mu>0$,
and \emph{time-like} if $a^\mu a_\mu<0$.

\end{itemize}

With these conventions, the quantity $s$ appearing in \eqref{W} denotes the
square of the total energy in the center-of-momentum frame,
\begin{equation}
\label{s}
s=s(p^\mu,q^\mu)\eqdef-(p^\mu+q^\mu)(p_\mu+q_\mu)
=2(-p^\mu q_\mu+1)\ge0,
\end{equation}
while $g$ denotes the relative momentum,
\begin{equation}
\label{g}
g=g(p^\mu,q^\mu)\eqdef
\sqrt{(p^\mu-q^\mu)(p_\mu-q_\mu)}
=\sqrt{2(-p^\mu q_\mu-1)}.
\end{equation}
In particular, $s$ and $g$ are related by $s=g^2+4$.
Using \eqref{lorentz.inner.prd}, we also have
\[
g=\sqrt{-(p^0-q^0)^2+|p-q|^2},
\]
where $|p-q|$ denotes the Euclidean distance in $\mathbb{R}^3$.
Unless otherwise stated, we write $g=g(p^\mu,q^\mu)$ and
$s=s(p^\mu,q^\mu)$.

The scattering angle $\theta$ is defined by
\[
\cos\theta=\frac{(p^\mu-q^\mu)(p'_\mu-q'_\mu)}{g^2},
\]
which is well defined; see \cite{GL1996}.
By the collision invariance
\begin{equation}
\label{collision.invariants}
p^\mu+q^\mu=p^{\prime\mu}+q^{\prime\mu}, \qquad \mu=0,1,2,3,
\end{equation}
we have $g(p^\mu,q^\mu)=g(p^{\prime\mu},q^{\prime\mu})$ and
$s(p^\mu,q^\mu)=s(p^{\prime\mu},q^{\prime\mu})$.

Finally, $C>0$ denotes a generic constant whose value may change from line to
line. We write $A\lesssim B$ if $A\le CB$ for some such constant $C$, and
$A\approx B$ if both $A\lesssim B$ and $B\lesssim A$ hold.

\subsubsection{Special norms for solutions}
\label{sec.norm}

\begin{itemize}

\item
For each fixed $x_1 \in [0,1]$ and $\ell \in \mathbb{R}$, we define
\[
\|f(x_1,\cdot)\|_{L^1_{p,\ell}}
\;\eqdef\;
\int_{\mathbb{R}^3} (p^0)^{\ell}\,|f(x_1,p)|\,dp.
\]

\item
Similarly, for each fixed $p \in \mathbb{R}^3$, we define
\[
\|f(\cdot,p)\|_{L^\infty_{x_1}}
\;\eqdef\;
\operatorname*{ess\,sup}_{x_1\in[0,1]} |f(x_1,p)|.
\]

\item
We use the mixed space
\[
f \in L^\infty_{x_1}\!\left([0,1];\, L^1_p\!\left(\mathbb{R}^3,\,(p^0)^{\frac12}\,dp\right)\right),
\]
that is,
\[
\|f\|_{L^\infty_{x_1}L^1_{p,\frac12}}
\;\eqdef\;
\operatorname*{ess\,sup}_{x_1\in[0,1]}
\int_{\mathbb{R}^3}
|f(x_1,p)|\,(p^0)^{\frac12}\,dp
< \infty.
\]

\item
Similarly, we define the mixed space
\[
f \in L^1_p\!\left(\mathbb{R}^3,\,(p^0)^{\frac12}\,dp;\, L^\infty_{x_1}([0,1])\right),
\]
that is,
\[
\|f\|_{L^1_{p,\frac12}L^\infty_{x_1}}
\;\eqdef\;
\int_{\mathbb{R}^3}
\operatorname*{ess\,sup}_{x_1\in[0,1]}
|f(x_1,p)|\,(p^0)^{\frac12}\,dp
< \infty.
\]

\item
For $k \ge 0$, let $\varphi(p) \eqdef e^{k p^0}$ and define
\[
\|f\|
\;\eqdef\;
\int_{\mathbb{R}^3}
\operatorname*{ess\,sup}_{x_1\in[0,1]}
|f(x_1,p)|\,(p^0)^{\frac12}\,\varphi(p)\,dp
< \infty.
\]
Then, in general, the following inequalities hold:
\begin{equation}\label{function space ineq}
\|f\|_{L^\infty_{x_1}L^1_{p,\frac12}}
\;\le\;
\|f\|_{L^1_{p,\frac12}L^\infty_{x_1}}
\;\le\;
\|f\|.
\end{equation}

\item The inverse-Laplacian norm is defined by
\begin{align*}
\|f\|_{-1}
&\eqdef 4\pi
\sup_{a\in\mathbb{R}^3}
(-\Delta_p)^{-1}
\bigl((p^0)^{1/2}\varphi(p)\operatorname*{ess\,sup}_{x_1\in[0,1]}
|f(x_1,p)|\bigr)(a) \\
&=
\sup_{a\in\mathbb{R}^3}
\int_{\mathbb{R}^3}
\frac{(p^0)^{1/2}\varphi(p)}{|p-a|}
\|f(\cdot,p)\|_{L^\infty_{x_1}}\,dp .
\end{align*}

\item Finally, the hyperplane norm is defined by
\[
\|f\|_{\mathrm{hyp}}
\eqdef
\sup_E
\int_E
(p^0)^{1/2}\varphi(p)
\|f(\cdot,p)\|_{L^\infty_{x_1}}\,d\sigma_E(p),
\]
where the supremum is taken over all two-dimensional planes $E\subset\mathbb{R}^3$
passing through the origin and $d\sigma_E$ denotes the Lebesgue
measure on $E$.
\end{itemize}

%
%
%
%
%

\subsection{Center-of-momentum framework}

There are several equivalent ways to perform the Dirac delta integrations
in the collision operator \eqref{Q original}.
In the \emph{center-of-momentum} (or center-of-mass) frame, the gain and loss
terms of the relativistic Boltzmann collision operator admit the following
representations; see \cite{DeGroot,MR2765751}:
\begin{align}
\label{Q center of momentum}
\begin{split}
Q^+(f,h)
&=\int_{\mathbb{R}^3}\int_{\mathbb{S}^2}
v_{\textup{\o}}\,
\sigma(g,\theta)\,
f(p')\,h(q')\,d\omega\,dq,\\
Q^-(f,h)
&=\int_{\mathbb{R}^3}\int_{\mathbb{S}^2}
v_{\textup{\o}}\,
\sigma(g,\theta)\,
f(p)\,h(q)\,d\omega\,dq,
\end{split}
\end{align}
where $\sigma(g,\theta)$ denotes the scattering kernel and
$v_{\textup{\o}}$ is the M{\o}ller velocity, given by
\begin{equation}
\label{mollerV}
v_{\textup{\o}}=v_{\textup{\o}}(p,q)
=\sqrt{\left|\frac{p}{p^0}-\frac{q}{q^0}\right|^2
-\left|\frac{p}{p^0}\times\frac{q}{q^0}\right|^2}
=\frac{g\sqrt{s}}{p^0 q^0}.
\end{equation}

The pre-collisional momentum pair $(p,q)$ and the post-collisional pair
$(p',q')$ are related by
\begin{equation}
\label{p'}
\begin{split}
p'
&=\frac{p+q}{2}
+\frac{g}{2}
\left(
\omega+(\gamma-1)(p+q)\frac{(p+q)\cdot\omega}{|p+q|^2}
\right),\\
q'
&=\frac{p+q}{2}
-\frac{g}{2}
\left(
\omega+(\gamma-1)(p+q)\frac{(p+q)\cdot\omega}{|p+q|^2}
\right),
\end{split}
\end{equation}
where
\[
\gamma=\frac{p^0+q^0}{\sqrt{s}}.
\]
In particular,
\begin{equation}
\label{gamma minus one}
\gamma-1
=\frac{p^0+q^0-\sqrt{s}}{\sqrt{s}}
=\frac{|p+q|^2}{\sqrt{s}\,(p^0+q^0+\sqrt{s})}.
\end{equation}

The post-collisional energies are given by \cite{MR2765751}
\begin{equation}
\label{energypq}
p'^0
=\frac{p^0+q^0}{2}
+\frac{g}{2\sqrt{s}}\,
\omega\cdot(p+q),
\qquad
q'^0
=\frac{p^0+q^0}{2}
-\frac{g}{2\sqrt{s}}\,
\omega\cdot(p+q).
\end{equation}

With this formulation, the linearized collision operator $\mathcal{L}$ can be written as
\begin{equation}
\label{Lf}
\mathcal{L}f
=\int_{\mathbb{R}^3}\int_{\mathbb{S}^2}
v_{\textup{\o}}\,
\sigma(g,\theta)\,
f(q)\,d\omega\,dq.
\end{equation}

We now state our main results.

\subsection{Main results}

We consider the relativistic Boltzmann equation in a slab geometry subject to
non-negative inflow boundary conditions prescribed at $x_1=0$ and $x_1=1$.
We establish the existence and uniqueness of a steady-state solution that is
symmetric with respect to the transverse spatial variables $x_2$ and $x_3$.
The solution is constructed in a weighted $L^1_p L^\infty_{x_1}$ framework,
exhibits exponential decay in the momentum variable $p$, and satisfies an
intrinsic $(-\Delta_p)^{-1}$-type regularity estimate.

\begin{theorem}[Existence, uniqueness, and intrinsic regularity]
\label{thm:existence}
Suppose that the inflow boundary profile $f_{LR}(p)$ satisfies
\eqref{boundary condition}.
Let $\varphi(p)\eqdef e^{k p^0}$ for some constant $k>0$.
If $k$ is sufficiently large, then there exists a unique non-negative steady
state $f$ solving \eqref{recall} in the sense of the mild formulation
\eqref{f representation for p1}, such that the following weighted bounds hold:
\begin{enumerate}
\item
\[
\|f\|\eqdef
\int_{\mathbb{R}^3}
\operatorname*{ess\,sup}_{x_1\in[0,1]}
|f(x_1,p)|
\,(p^0)^{\frac12}\varphi(p)\,dp
<\infty,
\]
\item
\[
\|f\|_{-1}
\eqdef 4\pi
\sup_{a\in\mathbb{R}^3}
(-\Delta_p)^{-1}
\Bigl(
(p^0)^{1/2}\varphi(p)
\operatorname*{ess\,sup}_{x_1\in[0,1]}
|f(x_1,p)|
\Bigr)(a)
<\infty,
\]
that is,
\[
\|f\|_{-1}
=
\sup_{a\in\mathbb{R}^3}
\int_{\mathbb{R}^3}
\frac{(p^0)^{1/2}\varphi(p)}{|p-a|}
\,
\|f(\cdot,p)\|_{L^\infty_{x_1}}
\,dp
<\infty.
\]
\end{enumerate}
\end{theorem}
In addition, under a geometric integrability assumption on the inflow profile,
we establish a hyperplane integrability estimate for the steady state in
momentum space.  One of the main contributions of the present work is precisely
this hyperplane estimate, which appears to be genuinely new in the relativistic
setting.  The result provides a form of codimension--one control that goes
beyond the standard weighted $L^1_p$ framework.

\begin{theorem}[Hyperplane integrability of the steady state]
\label{thm:hyperplane}
Let $f$ be the steady-state solution constructed in
Theorem~\ref{thm:existence}.
Assume in addition that the inflow boundary profile satisfies the hyperplane
bound $\|f_{LR}\|_{\mathrm{hyp}}<\infty$.
Then the steady state $f$ satisfies the uniform hyperplane integrability
estimate
\[
\|f\|_{\mathrm{hyp}}
\eqdef
\sup_E
\int_E
(p^0)^{1/2}\varphi(p)
\,
\|f(\cdot,p)\|_{L^\infty_{x_1}}
\,d\sigma_E(p)
<\infty,
\]
where the supremum is taken over all two-dimensional planes
$E\subset\mathbb{R}^3$ passing through the origin, and $d\sigma_E$ denotes the
Lebesgue measure on $E$.
\end{theorem}

\begin{remark}
The hyperplane integrability estimate in
Theorem~\ref{thm:hyperplane} provides a form of codimension--one control in
momentum space that is strictly stronger than the weighted $L^1_p$ bound in
Theorem~\ref{thm:existence}.  To the best of our knowledge, such a hyperplane
estimate has not been previously established for steady solutions of the
relativistic Boltzmann equation.  The proof relies on a Lorentz--geometric
reduction that exploits the invariance structure of the relativistic collision
operator and yields the estimate as a genuinely \emph{a posteriori} regularity
property of the steady solution.

Such codimension--one bounds are naturally adapted to the nonlocal structure of
the collision operator and may play a role in further qualitative analyses of
the steady state, including stability questions and limiting procedures.
\end{remark}

\subsection{Main difficulties and strategies}

The main difficulty of the present work is not merely the treatment of the
relativistic gain term, but the construction of a stationary solution theory on a
slab of fixed size,
\(
x_1\in[0,1],
\)
under non-perturbative inflow boundary conditions. In contrast to thin-slab
settings, no small geometric parameter is available here. Consequently, the
attenuation in the mild formulation cannot be produced by the spatial size of the
domain itself and must instead be generated intrinsically through the collision
frequency.

This point becomes transparent once the equation is written in the mild form.
Indeed, the solution operator involves the exponential factor
\[
1_{p_1>0}\exp\!\left(-\frac{1}{\hat p_1}\int \mathcal{L}f(z,p)\,dz\right),
\]
whose upper bound depends on a coercive lower bound for \(\mathcal{L}f\). This is highly
nontrivial for two reasons. First, the transport coefficient \(1/\hat p_1\) becomes
singular near grazing momentum \(p_1=0\). Second, since we do not work in a
perturbative regime around a Maxwellian equilibrium, the required lower bound
for \(\mathcal{L}f\) cannot be imported from a linearized spectral theory. It must be created
directly from the structure of the inflow problem itself.

For this reason, the boundary condition \eqref{boundary condition} is not merely an admissibility
assumption ensuring integrability and decay of the incoming data. Its deeper role
is to encode a genuine non-degeneracy of particles entering the slab from both
sides. This non-degeneracy is made quantitative in the new preliminary Lemma
1.6, where we prove a uniform positive lower bound for a convolution-type
integral generated by the boundary profiles \(f_L\) and \(f_R\). In particular, the
lemma shows that the incoming mass concentrated on the sets \(Q_{r_1,r_2}\) and
\(\overline Q_{r_1,r_2}\) produces a strictly positive contribution uniformly for all
\(p\in\mathbb R^3\). This is the key mechanism through which the boundary data
generate coercivity in the interior.

This lower-bound mechanism is then inserted into Proposition 2.1, where we
establish the two-sided estimate
\[
C_\ell (p^0)^{1/2}\le \mathcal{L}f(x_1,p)\le C_u (p^0)^{1/2}.
\]
The lower bound is the decisive part. It shows that the loss operator has the same
momentum growth as in the relativistic equilibrium setting, but here it is obtained
without any proximity to equilibrium. Rather, it is created from the inflow
structure together with the mild lower bound inherited from the transport
representation. In this sense, one of the main novelties of the analysis is that the
coercivity of the collision frequency is not postulated a priori and does not come
from a background Maxwellian; it is produced by the boundary inflow itself.

A second major difficulty is that the operator \(L\) enters the mild formulation
nonlinearly through the exponential attenuation factor. Therefore, to carry out the
Banach fixed-point argument, it is not enough to know that \(\mathcal{L}f\) is coercive. One
must also show that \(\mathcal{L}f\) depends stably on \(f\) in a sufficiently strong weighted
topology. This is achieved in Lemma 2.2, where we prove
\[
|\mathcal{L}f(x_1,p)-\mathcal{L}f(x_1,p)|
\lesssim e^{-\sqrt{k}/2}(p^0)^{1/2}\|f-h\|.
\]
The importance of this estimate is twofold. First, it provides quantitative control
of the variation of the nonlinear attenuation factor. Second, the exponentially
small prefactor \(e^{-\sqrt{k}/2}\) is what ultimately allows the solution operator to
become contractive for large \(k\). Thus the loss term plays a double role in the
proof: its coercivity stabilizes transport, while its continuity stabilizes the
nonlinear dependence of the mild formulation on the unknown solution.

The next difficulty concerns the relativistic gain operator \(Q^+\). Here the main
obstruction is not simply nonlocality, but the highly anisotropic post-collisional
geometry induced by the center-of-momentum representation, the Møller velocity,
and the relativistic transformation law for \((p',q')\). Because of this structure,
standard convolution arguments from the Newtonian theory do not directly apply.
To overcome this, we do not attempt to control \(Q^+\) in a single norm. Instead,
we build a mixed weighted framework in which different aspects of the operator
are handled by different but complementary estimates.

More precisely, we construct the solution space \(A\) using the weighted norm
\[
\|f\|
=
\int_{\mathbb R^3}(p^0)^{1/2}\phi(p)\,
\|f(\cdot,p)\|_{L^\infty_{x_1}}\,dp
\]
together with the inverse-Laplacian type quantity
\[
\|f\|_{-1}
=
\sup_{a\in\mathbb R^3}
\int_{\mathbb R^3}
\frac{(p^0)^{1/2}\phi(p)}{|p-a|}
\|f(\cdot,p)\|_{L^\infty_{x_1}}\,dp.
\]
The latter is not an auxiliary technical refinement. It is part of the core analytic
architecture of the proof. The reason is that the post-collisional geometry of
\(Q^+\) naturally generates kernel structures of potential type in momentum space (see \eqref{Ibound1}),
and these cannot be closed at the level of the weighted \(L^1\)-norm alone.

This is reflected in Proposition 3.1, where we prove the potential-type estimate
\[
\sup_{a\in\mathbb R^3}
\int_{\mathbb R^3}
\frac{\phi(p)}{|p-a|}
\|Q^+(f,f)(\cdot,p)\|_{L^\infty_{x_1}}\,dp
\lesssim \|f\|^2.
\]
This estimate reveals an intrinsic \((-\Delta_p)^{-1}\)-structure hidden in the
relativistic collision geometry. It is precisely this structure that motivates the norm
\(\|\cdot\|_{-1}\) in the first place and allows one to propagate inverse-Laplacian
regularity through the nonlinear iteration.

At the same time, the fixed-point argument also requires a pointwise weighted
bound on the gain term itself. This is obtained in Lemma 3.2:
\[
\int_{\mathbb R^3}\phi(p)\,
\|Q^+(f,f)(\cdot,p)\|_{L^\infty_{x_1}}\,dp
\lesssim k^{-1}\|f\|\,\|f\|_{-1}.
\]
The factor \(k^{-1}\) is essential here. Without such a small coefficient, the
nonlinear gain contribution would not be dominated by the damping created by the
loss term, and the self-map and contraction properties of the solution operator
would not close. Thus the gain analysis is not merely quantitative refinement: it is
the mechanism that balances the nonlinear production term against the coercive
transport attenuation.

A further contribution of the paper concerns codimension-one regularity in
momentum space (Theorem \ref{thm:hyperplane}). After the existence and uniqueness theory is completed, we
show that hyperplane integrability propagates from the inflow boundary to the
interior steady solution. This part of the analysis is logically separate from the
fixed-point construction. In particular, no a priori hyperplane norm is imposed in
the definition of the solution space \(A\). Instead, the hyperplane bound is derived
after the steady state has been constructed, as a genuinely a posteriori regularity
property.

The proof of this hyperplane estimate relies on a geometric reduction based on
Lorentz transformations. By rotating an arbitrary two-dimensional momentum
hyperplane into a canonical configuration, we reduce the estimate to a normalized
geometric setting in which the relativistic collision structure can be exploited more
effectively. This Lorentz-based reduction is therefore not just a geometric
convenience; it is the key device that makes uniform codimension-one integration
possible across all hyperplanes.

In summary, the proof is organized around three interconnected mechanisms.
First, the non-degenerate inflow boundary data generate a uniform lower bound for
the collision frequency through Lemma 1.6 and Proposition 2.1. Second, the mixed
weighted space \(A\), incorporating both the weighted \(L^1\)-norm and the
inverse-Laplacian norm, provides the correct setting for controlling the nonlinear
gain operator. Third, a Lorentz-geometric reduction yields the propagation of
hyperplane integrability as an additional a posteriori property of the steady state.
Together, these ingredients provide a non-perturbative framework for constructing
steady relativistic Boltzmann solutions on a full slab domain.

\subsection{Further prior results in relativistic collisional kinetic theory}

This section briefly reviews prior developments in relativistic collisional
kinetic theory and related areas.

We begin with classical works on relativistic hydrodynamics, including the
foundational contributions of Synge, Israel, Stewart, Landau, and others
\cite{synge1958relativistic,israel1963relativistic,israel1976thermodynamics,
israel1979transient,stewart1977transient,van1987chapman,PhysRev.58.269,
landau1987fluid,ruggeri1986relativistic,DeGroot,MR1898707}.
The study of global-in-time existence for the relativistic Boltzmann equation
was initiated by Dudyński and Ekiel--Jeżewska in 1988--89 through the analysis of
the linearized equation \cite{D-E3,D}, and was later extended to the nonlinear
setting by Glassey and Strauss in a series of works during the mid-1990s
\cite{MR1211782,GS4}.
For the relativistic analogue of the near-vacuum regime, we refer to
\cite{GL-Vacuum,MR2982812}.

The stability theory of the relativistic Maxwellian (J\"uttner) distribution
has been investigated for both hard and soft potentials in
\cite{MR2543323,MR1211782,MR2728733}, and was recently extended to the
non–angular-cutoff case in \cite{MR4470411}.
The Newtonian limit of the relativistic Boltzmann equation has been studied in
\cite{Cal,MR2679588}, while results on hydrodynamic limits can be found in
\cite{MR2793935}.
Blow-up phenomena for the relativistic Boltzmann equation without the loss term
were examined in \cite{MR2102321}.
Additional results on the Cauchy problem are available in
\cite{D-E0,D-E0er,D-E2,MR2378164}.
The spatially homogeneous quantum relativistic equation and its inhomogeneous
counterpart were studied in \cite{E-M-V} and \cite{MR4264953}, respectively.
Regularizing effects of the relativistic collision operator have also been
addressed in \cite{MR1402446,MR3880739}.

For the relativistic Vlasov--Maxwell--Landau system with self-consistent
electromagnetic fields, Guo and Strain established global existence and
stability near relativistic Maxwellian equilibria for the relativistic Landau
equation in \cite{Guo-Strain3}, and later extended their analysis to the
relativistic Vlasov--Maxwell--Boltzmann system in \cite{Guo-Strain2}.
More recently, local well-posedness \cite{doi:10.1137/23M1608938} and asymptotic
stability of the J\"uttner distribution \cite{2401.00554} have been proved for
the relativistic Vlasov--Maxwell--Landau system in bounded domains.
A detailed treatment of the relativistic collision operator in
center-of-momentum coordinates can be found in
\cite{MR2765751,MR4268835}.

Regarding the spatially homogeneous relativistic Boltzmann equation, various
moment estimates were obtained in
\cite{MR3166961,MR4271957,MR4156121}.
Entropy dissipation and global existence of weak solutions for the spatially
homogeneous relativistic Landau equation were studied in \cite{StrainTas},
together with refined moment propagation results.
Conditional singularity formation for weak solutions was investigated in
\cite{1903.05301}.
We also mention studies of Cauchy problems arising in cosmological models in
\cite{MR4222138,MR3620020,MR3169776}.

Another line of research concerns relaxation-time approximations of the
relativistic Boltzmann equation.
The Marle-type relativistic BGK model has been studied in
\cite{MR2988960,MR3300786,1801.08382}, while the Anderson--Witting model is
discussed in \cite{1811.10023}.
A recently proposed relativistic BGK model was introduced in \cite{Pennisi_2018},
and its existence theory was developed in \cite{H-R-Yun}.

For further general discussions on relativistic kinetic equations, we refer to
\cite{C,C-I-P,GL1996,Vil02,MR1898707,MR3189734}.
\subsection{Geometric and analytic preliminaries}\label{prelim}

In this subsection, we collect several analytic and geometric tools that will be
used repeatedly throughout the paper, particularly in the estimates of the
collision operator and in the derivation of hyperplane bounds.
We begin with a classical coercive inequality for the relativistic relative
momentum in the center-of-momentum framework.

\begin{lemma}[Coercive inequality for the relative momentum]
\label{coersive inequality}
The following coercive inequality for the relativistic relative momentum $g$
is taken from \cite[Lemma~3.1]{MR1211782}.
It holds that 
\begin{equation}\label{gINEQ}
\frac{\bigl(|p-q|^2+|p\times q|^2\bigr)^{1/2}}{\sqrt{p^0 q^0}}
\le g(p^\mu,q^\mu)
\le \min\{|p-q|,\,2\sqrt{p^0q^0}\}.
\end{equation}
\end{lemma}

\begin{remark}
In \cite{MR1211782}, the relative momentum is defined as one half of our definition
\eqref{g}.
\end{remark}

We next establish a uniform positive lower bound for the convolution-type
integrals associated with the boundary profiles $f_L$ and $f_R$.
The assumptions required here are weaker than those imposed in
\eqref{boundary condition}.
\begin{lemma}\label{lemma m}
Let $0<r_1<r_2<\infty$ and define
\[
Q_{r_1,r_2}\eqdef \{q\in\mathbb{R}^3:\ q_1\ge r_1 \text{ and } r_1\le |q|\le r_2\},
\] and its reflection \[
\bar{Q}_{r_1,r_2}\eqdef \{q\in\mathbb{R}^3:\ q_1\le -r_1 \text{ and } r_1\le |q|\le r_2\}.
\]Suppose $f_L$ satisfies
\[
f_L\ge 0 \quad \text{a.e. on }Q_{r_1,r_2},
\qquad 
\int_{Q_{r_1,r_2}} f_L(q)\,dq>0.
\] Assume similarly for $f_R$ the following on $\bar{Q}_{r_1,r_2}$
\[
f_R\ge 0 \quad \text{a.e. on }\bar{Q}_{r_1,r_2},
\qquad 
\int_{\bar{Q}_{r_1,r_2}} f_R(q)\,dq>0.
\]
Define
\[
I(p)\eqdef \int_{Q_{r_1,r_2}} c_0\frac{|p-q|^2}{(p^0q^0)^2}f_L(q)\,dq.
\]
Then there exists a constant $m>0$ such that
\[
I(p)\ge m
\qquad \text{for all }p\in\mathbb{R}^3.
\] The same estimate holds also for the integral $\bar{I}(p)$ with respect to the integrand $f_R$ which is defined as
\[
\bar{I}(p)\eqdef \int_{\bar{Q}_{r_1,r_2}} c_0\frac{|p-q|^2}{(p^0q^0)^2}f_R(q)\,dq.
\]
\end{lemma}

\begin{proof}For the sake of brevity, we prove the lower bound estimate for $I(p)$ only. The proof for $\bar{I}(p)$ is the same.  
Let $Q\eqdef Q_{r_1,r_2}$. Since $q\in Q$ implies $r_1\le |q|\le r_2$, we have
\[
1+r_1^2 \le (q^0)^2 = 1+|q|^2 \le 1+r_2^2
\qquad \text{for all }q\in Q.
\]
In particular, $q^0$ is bounded above and below by positive constants on $Q$.

We first prove that $I(p)$ is positive.
Fix $p\in\mathbb{R}^3$. Since $f_L\ge0$, the integrand is nonnegative and thus $I(p)\ge0$. 
If $I(p)=0$, then
\[
|p-q|^2 f_L(q)=0 \qquad \text{for a.e. }q\in Q.
\]
Hence $f_L(q)\neq0$ only when $q=p$. Since $\{p\}$ has Lebesgue measure zero, this contradicts
$\int_Q f_L(q)\,dq>0$. Therefore $I(p)>0$ for every $p\in\mathbb{R}^3$.

Now, we prove the continuity of $I(p)$.
Let $p_n\to p$. Since $Q$ is bounded, there exists $C>0$ such that
\[
0\le c_0\frac{|p_n-q|^2}{(p_n^0q^0)^2}f_L(q)\le C f_L(q)
\qquad \text{for }q\in Q.
\]
Because $f_L\in L^1(Q)$, dominated convergence yields
$
I(p_n)\to I(p).
$
Hence $I$ is continuous on $\mathbb{R}^3$.

Now, we study the behavior as $|p|\to\infty$.
For $q\in Q$,
\[
\frac{|p-q|^2}{(p^0q^0)^2}
=
\frac{|p|^2-2p\cdot q+|q|^2}{(1+|p|^2)(q^0)^2}.
\]
Since $Q$ is bounded, the terms involving $q$ vanish uniformly as $|p|\to\infty$, and thus
\[
\frac{|p-q|^2}{(p^0q^0)^2}\longrightarrow \frac{1}{(q^0)^2}
\qquad \text{uniformly for }q\in Q.
\]
Consequently,
\[
I(p)\longrightarrow 
L \eqdef c_0\int_Q \frac{f_L(q)}{(q^0)^2}\,dq
\qquad \text{as }|p|\to\infty.
\]
Since $f_L\ge0$ and $\int_Q f_L(q)\,dq>0$, we have $L>0$. Thus there exists $R>0$ such that
\[
|p|\ge R \implies I(p)\ge \frac{L}{2}.
\]

Finally, we make a uniform lower bound. 
On the compact set $\{|p|\le R\}$, the function $I$ is continuous and strictly positive.
Hence
\[
m_1\eqdef \min_{|p|\le R} I(p) >0.
\]
Let
$
m\eqdef \min\left\{m_1,\frac{L}{2}\right\}>0.
$ 
Then $I(p)\ge m$ for all $p\in\mathbb{R}^3$.
\end{proof}  
We now recall several elementary facts concerning Lorentz transformations
that will be used in the analysis of the collision operator and in the
geometric reduction underlying the hyperplane estimates.

A $4\times4$ real matrix $\Lambda=(\Lambda^\mu_{\ \nu})$ is called a (proper)
Lorentz transformation if $\det\Lambda=1$ and
\begin{equation}\label{lorentz.trans.inv}
\Lambda^\kappa_{\ \mu}\,\eta_{\kappa\lambda}\,\Lambda^\lambda_{\ \nu}
=\eta_{\mu\nu},
\qquad \mu,\nu=0,1,2,3.
\end{equation}
In particular, Lorentz transformations preserve the Lorentz inner product:
\begin{equation}\label{invariance}
p^\mu q_\mu
=
(\Lambda p)^\mu(\Lambda q)_\mu .
\end{equation}
Any Lorentz transformation is invertible, and its inverse is again a Lorentz
transformation. We refer to
\cite{MR2707256,MR2728733,MR2679588,MR2793935}
for further background.

A key geometric ingredient in our analysis is the ability to normalize the
orientation of momentum-space hyperplanes via Lorentz transformations. 
Let $a=(a_1,a_2,a_3)\in\mathbb{R}^3$. There exists a Lorentz transformation
$\Lambda$, given by a spatial rotation, such that
\begin{equation}\label{rot.lorentz}
\Lambda (0,a_1,a_2,a_3)=(0,0,0,|a|),
\end{equation}
where $|a|=(a_1^2+a_2^2+a_3^2)^{1/2}$.
In particular, $\Lambda$ preserves the time component and acts as a rotation on
the spatial momentum variables.

This geometric reduction allows us to rotate arbitrary two-dimensional
hyperplanes in momentum space into a canonical configuration.
It plays a crucial role in the proof of the hyperplane estimate
(Proposition~\ref{prop.3.4}) and in its propagation to the steady-state
solution established in Section~\ref{sec.6}.

\subsection{Outline of the article}

The remainder of this article is organized as follows.
In Section~\ref{sec.L}, we establish quantitative estimates for the loss term operator
$\mathcal{L}$, including uniform upper and lower bounds and continuity-type estimates.
These results provide the coercivity and stability properties of the collision
frequency that will be repeatedly used in the fixed-point framework.

Section~\ref{egain} is devoted to estimates for the gain term $Q^+$.
We derive two complementary bounds, a potential-type convolution estimate and a
pointwise weighted estimate, which together
allow us to control the nonlinear structure of the relativistic collision
operator in the weighted setting.

In Section~\ref{sec.main}, we introduce the solution space $\mathcal{A}$ and define the
solution operator $A$ through the mild formulation of the stationary problem.
We then prove the invariance (self-map) property $A\mathcal{A}\subset
\mathcal{A}$, which supplies the a priori bounds needed for the Banach
fixed-point argument.

In Section~\ref{sec.5}, we establish the contraction property of $A$ on $\mathcal{A}$
and conclude the existence and uniqueness of a non-negative stationary mild
solution, thereby proving Theorem~\ref{thm:existence}.

Finally, Section~\ref{sec.6} proves an additional \emph{a posteriori} regularity statement (Theorem \ref{thm:hyperplane}). We first establish an additional uniform hyperplane estimate on the gain term $Q^+.$ 
Assuming the inflow boundary profile satisfies a uniform hyperplane integrability
bound, we show that the corresponding stationary solution enjoys the same
hyperplane estimate. This propagation result is independent of the existence and
uniqueness theory and provides further structural insight into steady solutions.

%
%
%
%
%

\section{Coercivity and Continuity Estimates for the Loss Operator}
\label{sec.L}In the following proposition, we first establish a uniform estimate for the loss term operator $\mathcal{L}$; the following is a generalized version of Proposition 1. (i) of \cite[page 8]{2411.06533} (special case for the following theorem in the case that $f=J_\infty$ for the relativistic J\"uttner equilibrium $J_\infty(p)\eqdef \frac{1}{4\pi cT_\infty K_2(c^2/T_\infty)}e^{-\frac{p_0u_{\infty,0}-p_1u_{\infty,1}}{T_\infty}}$).
\begin{proposition}[Upper- and lower-bound estimates for $\mathcal{L}f$]\label{lowerL} Let a given boundary profile $f_{LR}$ satisfy \eqref{boundary condition}.  Let $f\in L^\infty_{x_1}\left([0,1]; L^1_p\!\left(
\mathbb{R}^3,\,(p^0)^{\frac12}\,dp\,
\right)\right)$ satisfy \begin{equation}
    \label{lower assumption}\mathcal{L}f\ge c_1\textup{ and }\|f\|_{L^\infty_{x_1}L^1_{p,\frac{1}{2}}}\le M,
\end{equation} for some constant $c_1$ and $M$.  Furthermore, suppose that $f$ satisfies 
\begin{equation}
    \label{f coercivity assumption} 
        f(x_1,p)\ge 1_{p_1>0}f_L(p)\exp\left(-\frac{c_1 x_1}{\hat{p}_1}\right)
        +1_{p_1<0}f_R(p)\exp\left(\frac{(1-x_1)c_1}{\hat{p}_1}\right).
    \end{equation}
Then there exist positive constants $C_{\ell}$ and $C_u$ that depend on $c_1$, $M$ and  $f_{LR}$ such that 
$$
C_{\ell} (\pZ)^{1/2} \leq 
(\mathcal{L}f)(x_1,p)=\int_\rth\int_{\mathbb{S}^2}v_{\textup{\o}} \sigma(g,\theta)f(x_1,q)d\omega dq
\leq C_u (\pZ)^{1/2}.
$$
\end{proposition}

\begin{proof}
    For the upper-bound, we observe
    \begin{multline*}
        (\mathcal{L}f)(x_1,p)=\int_\rth\int_{\mathbb{S}^2}v_{\textup{\o}} \sigma(g,\theta)f(x_1,q)d\omega dq
        \lesssim \int_\rth\int_{\mathbb{S}^2}
(p^0q^0)^{1/2}f(x_1,q)\sigma_0(\theta)d\omega dq\\*
\lesssim (p^0)^{1/2}\|f(x_1,\cdot)\|_{L^1_{p,1/2}}\lesssim M(p^0)^{1/2} ,
    \end{multline*}using \eqref{angular assumption}, \eqref{gINEQ}, and the fact that  $v_{\textup{\o}}\lesssim 1.$

    For the lower-bound we observe that
     \begin{multline*}
        (\mathcal{L}f)(x_1,p)=\int_\rth\int_{\mathbb{S}^2}v_{\textup{\o}} \sigma(g,\theta)f(x_1,q)d\omega dq\approx \int_\rth\int_{\mathbb{S}^2}\frac{g^2\sqrt{s}}{p^0q^0}f(x_1,q)\sigma_0(\theta)d\omega dq\\
        \gtrsim \int_\rth\int_{\mathbb{S}^2}
\frac{|p-q|^3 }{(p^0q^0)^{5/2}} f(x_1,q)\sigma_0(\theta)d\omega dq\approx c_0\int_\rth\int_{\mathbb{S}^2}
\frac{|p-q|^3 }{(p^0q^0)^{5/2}} f(x_1,q)dq,
    \end{multline*}by \eqref{angular assumption} using $\sqrt{s}\ge g \gtrsim \frac{|p-q|}{\sqrt{p^0q^0}}.$ and use the spherical coordinate for $q$ so that the angle between $p$ and $q$ is $\varphi.$ 
    Note that we further have
    \begin{equation}\label{pqp0q0}
        |p-q|\ge |p^0-q^0|\ge p^0-q^0,
    \end{equation}because 
    \begin{multline*}
        |p^0-q^0|=|\sqrt{1+|p|^2}-\sqrt{1+|q|^2}|
        =\frac{||p|^2-|q|^2|}{\sqrt{1+|p|^2}+\sqrt{1+|q|^2}}\\
        =\frac{||p|+|q||||p|-|q||}{\sqrt{1+|p|^2}+\sqrt{1+|q|^2}}\le ||p|-|q||\ \le|p-q|.
    \end{multline*}Therefore, we first have
    \begin{multline}\label{Lf lower 2}
        (\mathcal{L}f)(x_1,p)\ge c_0(p^0)^{1/2}\int_\rth
\frac{1}{(q^0)^{5/2}} f(x_1,q) dq
-\frac{c_0}{(p^0)^{5/2}}\int_\rth
(q^0)^{1/2}  f(x_1,q) dq.
    \end{multline}
Using \eqref{lower assumption} and \eqref{f coercivity assumption} in \eqref{Lf lower 2}, we have
    \begin{align*}\notag
        (\mathcal{L}f)(x_1,p)&\ge c_0(p^0)^{1/2}\int_\rth \frac{dq}{(q^0)^{5/2}}\bigg( 
  1_{q_1>0}f_L(q)\exp\left(-\frac{c_1x_1}{\hat{q}_1}\right)
        \\*&\quad+1_{q_1<0}f_R(q)\exp\left(\frac{c_1(1-x_1)}{\hat{q}_1}\right)\bigg)
-\frac{c_0}{(p^0)^{5/2}}M\\
&\ge c_0(p^0)^{1/2}\int_\rth \frac{dq}{(q^0)^{5/2}}\bigg( 
  1_{q_1>0}f_L(q)e^{\left(-\frac{c_1}{\hat{q}_1}\right)}
        +1_{q_1<0}f_R(q)e^{\left(\frac{c_1}{\hat{q}_1}\right)}\bigg)-c_0M\\&=c_0 \bar{c} (p^0)^{1/2}-c_0M,
    \end{align*}for some constant $\bar{c}>0$ that is determined by $f_{LR}$ and $c_1$ via \eqref{boundary condition}. Then together with $\mathcal{L}f\ge c_1$ in \eqref{lower assumption}, we have
    \begin{multline*}
        (A_1+A_2)\mathcal{L}f \ge A_1c_1 + A_2(c_0 \bar{c} (p^0)^{1/2}-c_0M)
        \ge A_2c_0\bar{c}(p^0)^{1/2} +(A_1c_1 -A_2c_0M).
    \end{multline*} By choosing $A_1$ sufficiently large and $A_2$ sufficiently small such that 
    $$A_1c_1 -A_2c_0M\ge 0,$$ we can conclude that 
    $$\mathcal{L}f \ge \frac{A_2c_0\bar{c}}{A_1+A_2}(p^0)^{1/2} .$$
    This completes the proof.
\end{proof}
Also, we have the following continuity-type estimate for $\mathcal{L}$:
\begin{lemma}
    \label{L.cont}
    For each $x_1\in [0,1]$, we have 
    $$|(\mathcal{L}f-\mathcal{L}h)(x_1,p)|\lesssim \min \{(p^0)^{1/2} \|f-h\|_{L^1_{p,1/2}},e^{-\frac{k}{\sqrt{2}}}(p^0)^{1/2}\|f-h\|\}.$$
\end{lemma}
\begin{proof}
    Using the center-of-momentum representation \eqref{Lf} of $L,$ we have 
    $$(\mathcal{L}f-\mathcal{L}h)(x_1,p)=\int_\rth\int_{\mathbb{S}^2} v_{\textup{\o}}\sigma(g,\theta)(f(x_1,q)-h(x_1,q))~d\omega dq.$$ Using $v_{\textup{\o}}\lesssim 1$ and $$\int_{\mathbb{S}^2}\sigma(g,\theta)d\omega\lesssim  g  \lesssim (p^0q^0)^{1/2},$$ we obtain for each $x_1\in [0,1],$   $$|(\mathcal{L}f-\mathcal{L}h)(x_1,p)|\lesssim (p^0)^{1/2} \|f-h\|_{L^1_{p,1/2}}.$$

    Alternatively, if we use, for every $q\in\rth$, that $$\varphi(q)^{-1}(p^0q^0)^{1/2}= e^{-k\sqrt{1+|q|^2}}(p^0q^0)^{1/2} \le e^{-\frac{k}{\sqrt{2}}(1+|q|)}(p^0q^0)^{1/2},$$  then we have
     \begin{multline*}
         |(\mathcal{L}f-\mathcal{L}h)(x_1,p)|\lesssim e^{-\frac{k}{\sqrt{2}}}(p^0)^{1/2} \int_\rth (q^0)^{1/2} e^{-\frac{k}{\sqrt{2}}|q|}\varphi(q)|f(x_1,q)-h(x_1,q)|dq\\
     \lesssim  e^{-\frac{k}{\sqrt{2}}}(p^0)^{1/2}\int_\rth \varphi(q)|f(x_1,q)-h(x_1,q)|dq\lesssim  e^{-\frac{k}{\sqrt{2}}}(p^0)^{1/2} \|f-h\|.
     \end{multline*}  This completes the proof.
\end{proof}
This completes our discussion of the loss term operator $\mathcal{L}$.   In the next section we will prove estimates for $Q^+(f,h)$ from \eqref{Q original} for an arbitrary function $f,h\ge 0$.

\section{Potential and Pointwise Estimates for the Gain Operator}
\label{egain}
In this section, we provide two necessary estimates for the gain term of the collision operator: potential-type and pointwise estimates. 
Note that throughout this section we assume that the arbitrary non-negative measurable function $f$ is not necessarily a solution to \eqref{RBE}.

\subsection{Potential-type estimate of $Q^+$.}

We first prove a potential-type estimate for the relativistic collision operator:\begin{proposition}\label{lemma1}
Let $\varphi(p) \eqdef e^{k p^0}$ for some $k>0$.  
Then, for any $a \in \mathbb{R}^3$, there holds
\[
\int_{\mathbb{R}^3}
\frac{\varphi(p)}{|p-a|}
\, \|Q^+(f,f)(\cdot,p)\|_{L^\infty_{x_1}}
\, dp
\;\lesssim\;
\|f\|^2,
\]
with the implicit constant independent of $a$.
\end{proposition}

\begin{proof}
We recall the definition of the $Q^+$ term in \eqref{Q original} as follows:
\begin{multline*}
I_a\eqdef \int_\rth\frac{\varphi(p)}{|p-a|}\|Q^+(f,f)(\cdot,p)\|_{L^\infty_{x_1}}dp
\\
\le  \int_{\rth}dp\ \frac{\varphi(p)}{|p-a|}\int_{\mathbb{R}^3}dq\int_{\mathbb{S}^2}d\omega\ v_{\textup{\o}}g \sigma_0(\theta)\|f(\cdot,p')\|_{L^\infty_{x_1}}\|f(\cdot,q')\|_{L^\infty_{x_1}},
\end{multline*}where we recall that $v_{\textup{\o}}=\frac{g\sqrt{s}}{p^0q^0}.$ 
Taking a pre-post relabelling of the variables $(p^{\mu}, q^{\mu}) \to (p^{\prime\mu}, q^{\prime\mu})$ with the Jacobian determinant $\frac{p^0q^0}{p'^0q'^0}$ and using the fact that  $s$ and $g$ are invariant under this transformation, we obtain that the integral $I_a$ is bounded from above as
$$
I_a
\le  \int_{\rth}dp\ \|f(\cdot,p)\|_{L^\infty_{x_1}}\int_{\mathbb{R}^3}dq\ \|f(\cdot,q)\|_{L^\infty_{x_1}}v_{\textup{\o}}g\int_{\mathbb{S}^2}d\omega\  \frac{\varphi(p')}{|p'-a|}\sigma_0(\theta).
$$
Now we use the representation \eqref{p'} of $p'$ and observe that
\begin{multline*}
    I_a
\le  \int_{\rth}dp\ \|f(\cdot,p)\|_{L^\infty_{x_1}}\int_{\mathbb{R}^3}dq\ \|f(\cdot,q)\|_{L^\infty_{x_1}}v_{\textup{\o}}g\\\times \int_{\mathbb{S}^2}d\omega\ \sigma_0(\theta) \frac{\varphi(p')}{\left|\frac{p+q}{2}+\frac{g}{2}\left(\omega+(\gamma-1)(p+q)\frac{(p+q)\cdot\omega}{|p+q|^2}\right)-a\right|}.
\end{multline*}
We then observe from $\displaystyle|x|=\max_{|y|=1}|y^{\top}x|$ that for any $a'\in\rth$,
\begin{align*}
&\left|\omega+(\gamma-1)(p+q)\frac{(p+q)\cdot \omega}{|p+q|^2}-a'\right|\\
&=\big|\left(I+(\gamma-1)A\otimes A\right)\omega-a' \big|\cr
&=\big|\left(I+(\gamma-1)A\otimes A\right)(\omega-a'') \big|\cr
&=\max_{|y|=1}\big|y^{\top}\left(I+(\gamma-1)A\otimes A\right)(\omega-a'')\big|\cr
&\geq\left|\left(\frac{\omega-a''}{|\omega-a''|}\right)^{\top}\left(I+(\gamma-1)A\otimes A\right)(\omega-a'')\right|\cr
&=\frac{1}{|\omega-a''|}\left(|\omega-a''|^2+(\gamma-1)\left\{A\cdot (\omega-a'')\right\}^2\right)\\*
&\geq |\omega-a''|,
\end{align*}where $\otimes$ is used to express the outer product $A(A\cdot \omega)=(A\otimes A) \omega$ and we denote $a''=\left(I+(\gamma-1)A\otimes A\right)^{-1}a'$ and $A=\frac{p+q}{|p+q|}$, and also the last inequality holds because $$\gamma-1=\frac{p^0+q^0-\sqrt{s}}{\sqrt{s}}=\frac{|p+q|^2}{\sqrt{s}(p^0+q^0+\sqrt{s})}>0.$$ The inverse $\left(I+(\gamma-1)A\otimes A\right)^{-1}$ exists, since $I+(\gamma-1)A\otimes A$ is positive definite. By choosing $a'=\frac{2}{g}\left(a-\frac{p+q}{2}\right)$, we have \begin{equation}\begin{split}\label{Ia estimate}
I_a
&\le  \int_{\rth}dp\int_{\mathbb{R}^3}dq\int_{\mathbb{S}^2}d\omega\ v_{\textup{\o}}g \sigma_0(\theta)\frac{\|f(\cdot,p)\|_{L^\infty_{x_1}}\|f(\cdot,q)\|_{L^\infty_{x_1}}\varphi(p')}{\left|\frac{p+q}{2}+\frac{g}{2}\left(\omega+(\gamma-1)(p+q)\frac{(p+q)\cdot\omega}{|p+q|^2}\right)-a\right|}\\
&=\int_{\rth}dp\int_{\mathbb{R}^3}dq\int_{\mathbb{S}^2}d\omega\ v_{\textup{\o}}g \sigma_0(\theta)\frac{\|f(\cdot,p)\|_{L^\infty_{x_1}}\|f(\cdot,q)\|_{L^\infty_{x_1}}\varphi(p')}{\frac{g}{2}\left|\omega+(\gamma-1)(p+q)\frac{(p+q)\cdot\omega}{|p+q|^2}-a'\right|}\\
&\le \frac{1}{4\pi}\int_{\rth}dp\ \|f(\cdot,p)\|_{L^\infty_{x_1}}\int_{\mathbb{R}^3}dq\ \|f(\cdot,q)\|_{L^\infty_{x_1}}v_{\textup{\o}}g\int_{\mathbb{S}^2}d\omega\  \frac{2\varphi(p')}{g|\omega-a''|},
\end{split}\end{equation}where we used $\sigma_0(\theta)\le \frac{1}{4\pi}$ by the hypothesis in Section \ref{sec.ang.hypo}. Now we further calculate the integral:
$$
\int_{\mathbb{S}^2} \frac{d\omega}{|\omega - a''|},
$$
where $\omega \in \mathbb{S}^2$ (a point on the unit sphere) and $a'' \in \mathbb{R}^3$ is a fixed vector.
   Without loss of generality, we can place the point $a''$ at a distance $r = |a''|$ from the origin, because the integral’s value will depend only on the distance $r$ due to spherical symmetry. Write $ a'' = (0, 0, r) $ in spherical coordinates.
    For any point $ \omega = (\sin \theta \cos \phi, \sin \theta \sin \phi, \cos \theta) $ on $\mathbb{S}^2$, we have
   $$
   |\omega - a''| = \sqrt{(\sin \theta \cos \phi)^2 + (\sin \theta \sin \phi)^2 + (\cos \theta - r)^2}.
   $$
   This expression simplifies to
   $$
   |\omega - a''| = \sqrt{1 + r^2 - 2r \cos \theta}.
   $$
   The surface measure on $\mathbb{S}^2$ is $ d\omega = \sin \theta \, d\theta \, d\phi $, and the integral becomes
   $$
   \int_{\mathbb{S}^2} \frac{d\omega}{|\omega - a''|} = \int_0^{2\pi} \int_0^{\pi} \frac{\sin \theta \, d\theta \, d\phi}{\sqrt{1 + r^2 - 2r \cos \theta}}.
   $$
   The integrand does not depend on $\phi$, and we can factor out the $\phi$-integral:
   $$
   \int_{\mathbb{S}^2} \frac{d\omega}{|\omega - a''|} = 2\pi \int_0^{\pi} \frac{\sin \theta \, d\theta}{\sqrt{1 + r^2 - 2r \cos \theta}}.
   $$ Finally, we make a change of variables $\theta\mapsto z=\cos\theta$. Then we observe that 
 $$
   \int_{\mathbb{S}^2} \frac{d\omega}{|\omega - a''|} = 2\pi \int_0^{\pi} \frac{\sin \theta \, d\theta}{\sqrt{1 + r^2 - 2r \cos \theta}}=2\pi \int_{-1}^1 \frac{dz}{\sqrt{1 + r^2 - 2rz}}.
   $$ By denoting $A= 1+r^2$ and $B=2r$ and by making another change of variables $z\mapsto u\eqdef \sqrt{A-Bz}$ such that $2udu=-Bdz,$ we have
\begin{multline*}
   \int_{\mathbb{S}^2} \frac{d\omega}{|\omega - a''|} = 2\pi \int_{-1}^1 \frac{dz}{\sqrt{A - Bz}}=2\pi \int_{\sqrt{A-B}}^{\sqrt{A+B}} \frac{2}{B}\frac{udu}{u} = \frac{4\pi}{B}(\sqrt{A+B}-\sqrt{A-B})\\*=\frac{2\pi}{|a''|}(1+|a''|-|1-|a''||).
\end{multline*} If $|a''|\ge 1,$ then we have 
$$ \int_{\mathbb{S}^2} \frac{d\omega}{|\omega - a''|} =\frac{2\pi}{|a''|}(1+|a''|-(|a''|-1))=\frac{4\pi}{|a''|}\le 4\pi.$$
On the other hand, if $|a''|<1,$ then we have
$$ \int_{\mathbb{S}^2} \frac{d\omega}{|\omega - a''|} =\frac{2\pi}{|a''|}(1+|a''|-(1-|a''|))= 4\pi.$$
Therefore, in any case, we have
$$ \frac{1}{4\pi}\int_{\mathbb{S}^2} \frac{d\omega}{|\omega - a''|} \le 1.$$
Plugging this back into \eqref{Ia estimate} and using $v_{\textup{\o}}\lesssim 1$, $g\lesssim \sqrt{p^0q^0}$, and $p'^0\le p^0+q^0$ such that $\varphi(p')\le \varphi(p)\varphi(q)$, we obtain
$$I_a
\lesssim  \int_{\rth}dp\ (p^0)^{1/2}\varphi(p)\|f(\cdot,p)\|_{L^\infty_{x_1}}\int_{\mathbb{R}^3}dq\ (q^0)^{1/2}\varphi(q)\|f(\cdot,q)\|_{L^\infty_{x_1}}\approx \| f\|^2.$$
This completes the proof.
\end{proof}




%
%
%
%

\subsection{Pointwise estimate}
In this section, we establish an additional weighted pointwise 
 estimate as in the following lemma.
\begin{lemma}\label{lemma.pointwise.gain of weight}
Let $\varphi(p) \eqdef e^{k p^0}$ for some $k>0$.  
Then the following estimate holds:
\[
\int_{\mathbb{R}^3}
\varphi(p)\,
\|Q^+(f,h)(\cdot,p)\|_{L^\infty_{x_1}}
\, dp
\;\lesssim\;
k^{-1}\,\|f\|_{-1}\,\|h\|.
\]
\end{lemma}

\begin{proof}Using the representation \eqref{Q original} and \eqref{W} for $Q^+$ we first observe that
\begin{equation}\begin{split}\label{initial I}
I&\eqdef \int_\rth \varphi(p)\|Q^+(f,h)(\cdot,p)\|_{L^\infty_{x_1}}dp
\\
&\le  \frac{1}{2}\int_{\rth}\frac{dp}{\pZ }\int_{\mathbb{R}^3}\frac{dq}{{\qZ }}\int_{\mathbb{R}^3}\frac{dq'\hspace{1mm}}{{\qZp }}\int_{\mathbb{R}^3}\frac{dp'\hspace{1mm}}{{\pZp }}  \varphi(p)sg \sigma_0(\theta) \delta^{(4)}(p^\mu +q^\mu -p'^\mu -q'^\mu)\\&\qquad\times \|f(\cdot,p')\|_{L^\infty_{x_1}}\|h(\cdot,q')\|_{L^\infty_{x_1}}.
\end{split}\end{equation}
Applying a pre-post relabelling of the variables $(p^{\mu}, q^{\mu}) \to (p^{\prime\mu}, q^{\prime\mu})$ and using the fact that  $s$ and $g$ are invariant under this transformation, we obtain that the integral $I$ is less than or equal to
\begin{multline*}
\frac{M}{2} \int_{\rth}\frac{dp}{\pZ }\int_{\mathbb{R}^3}\frac{dq}{{\qZ }}sg \|f(\cdot,p)\|_{L^\infty_{x_1}}\|h(\cdot,q)\|_{L^\infty_{x_1}}\\*\times \int_{\mathbb{R}^3}\frac{dq'}{{\qZp }}\int_{\mathbb{R}^3}\frac{dp'}{{\pZp }}e^{kp'^0} \delta^{(4)}(p^\mu +q^\mu -p'^\mu -q'^\mu),
\end{multline*}
 where we used $\sigma_0(\theta)\le M.$ Using \eqref{energypq}, we further have
\begin{multline*}
I\lesssim \int_{\rth}\frac{dp}{\pZ }\int_{\mathbb{R}^3}\frac{dq}{{\qZ }}sg \|f(\cdot,p)\|_{L^\infty_{x_1}}\|h(\cdot,q)\|_{L^\infty_{x_1}}\\\times\int_{\mathbb{R}^3}\frac{dq'}{{\qZp }}\int_{\mathbb{R}^3}\frac{dp'}{{\pZp }}e^{k(\frac{\pZ +\qZ }{2}+\frac{g}{2\sqrt{s}}\omega\cdot(p+q))} \delta^{(4)}(p^\mu +q^\mu -p'^\mu -q'^\mu).
\end{multline*}
The standard reduction to the \textit{center-of-momentum} frame representation \eqref{Q center of momentum} then gives that
$$I\lesssim  \int_{\rth}dp\int_{\mathbb{R}^3}dq \ v_{\textup{\o}} g \|f(\cdot,p)\|_{L^\infty_{x_1}}\|h(\cdot,q)\|_{L^\infty_{x_1}}\int_{\mathbb{S}^2}d\omega\ e^{k(\frac{\pZ +\qZ }{2}+\frac{g}{2\sqrt{s}}\omega\cdot(p+q))} .
$$

Note that 
\begin{align*}
   & \int_{\mathbb{S}^2}d\omega\ e^{k\frac{g}{2\sqrt{s}}\omega\cdot(p+q)}=\int_0^{2\pi} d\phi \int_0^\pi d\theta\ \sin\theta \ e^{k\frac{g}{2\sqrt{s}}|p+q|\cos\theta}\\
    &= \frac{4\pi\sqrt{s}}{kg|p+q|}(e^{k\frac{g}{2\sqrt{s}}|p+q|}-e^{-k\frac{g}{2\sqrt{s}}|p+q|})=\frac{8\pi\sqrt{s}}{kg|p+q|}\sinh\left(\frac{kg}{2\sqrt{s}}|p+q|\right)\\
    &\le \frac{8\pi\sqrt{s}}{kg|p+q|}e^{\frac{k}{2}(p^0+q^0)},
\end{align*}where we used $|p+q|\le p^0+q^0$ and $g\le \sqrt{s}$ for the last equality. Now using $\sqrt{s} \lesssim (p^0q^0)^{1 /2},$ we obtain
\begin{multline}\label{Ibound1}
    I\lesssim \int_{\rth}dp\int_{\mathbb{R}^3}dq \  v_{\textup{\o}}\frac{ (p^0q^0)^{1/2}}{k|p+q|}\|f(\cdot,p)\|_{L^\infty_{x_1}}\|h(\cdot,q)\|_{L^\infty_{x_1}} e^{k(\pZ +\qZ )}
   \lesssim  k^{-1}\|  f\|_{-1}\|  h\|,
\end{multline} 
since the norms are defined as
$$\|  f\|_{-1}= \sup_{a\in \rth}\int_{\rth}dp\ \frac{ (p^0)^{1/2}  }{|p-a|}\|f(\cdot,p)\|_{L^\infty_{x_1}}\varphi(p),$$ and $$\| f\|= \int_{\rth}dp\ (p^0)^{1/2}  \|f(\cdot,p)\|_{L^\infty_{x_1}}\varphi(p).$$ 
\end{proof}
In addition, we obtain the following corollary.
\begin{corollary}\label{cor.pointwise.gain of weight}
Let $\varphi(p) \eqdef e^{k p^0}$ for some $k>0$.  
Then we have
\[
\int_{\mathbb{R}^3}
\varphi(p)\,
\|Q^+(f,h)(\cdot,p)\|_{L^\infty_{x_1}}
\, dp
\;\lesssim\;
k^{-1}\,\|h\|_{-1}\,\|f\|.
\]
\end{corollary}

\begin{proof}
    By making the change of variables $(p',q',\omega)\mapsto (q',p',-\omega)$ at the level of \eqref{initial I}, we can have 
\begin{multline}\notag
 \int_\rth \varphi(p)\|Q^+(f,h)(\cdot,p)\|_{L^\infty_{x_1}}dp
\\
\le  \frac{1}{2}\int_{\rth}\frac{dp}{\pZ }\int_{\mathbb{R}^3}\frac{dq}{{\qZ }}\int_{\mathbb{R}^3}\frac{dp'\hspace{1mm}}{{\pZp }}\int_{\mathbb{R}^3}\frac{dq'\hspace{1mm}}{{\qZp }}  \varphi(p)sg \sigma_0(\pi-\theta) \delta^{(4)}(p^\mu +q^\mu -p'^\mu -q'^\mu)\\\times \|f(\cdot,q')\|_{L^\infty_{x_1}}\|h(\cdot,p')\|_{L^\infty_{x_1}}.
\end{multline} Since we have $\sigma_0(\pi-\theta)\le M,$ the rest of the proof follows the same as that of Lemma \ref{lemma.pointwise.gain of weight} with $f$ and $h$ swapped. This gives the desired result.
\end{proof}

\section{Construction of the Solution Operator}
\label{sec.main}
For the proof of the existence of a unique mild solution to the stationary problem, we will use the Banach fixed-point theorem applied in the following space:
\begin{multline}\label{sol space}
    \mathcal{A}=\bigg\{f \in L^\infty_{x_1}\left([0,1];L^1_p\!\left(\mathbb{R}^3,\,(p^0)^{1/2}\,dp\right)\right)
:f\ge 0,\ \mathcal{L}f\ge c_1, \ \\ 
        f(x_1,p)\ge 1_{p_1>0}f_L(p)\exp\left(-\frac{c_1 x_1}{\hat{p}_1}\right)
        +1_{p_1<0}f_R(p)\exp\left(\frac{(1-x_1)c_1}{\hat{p}_1}\right),\\\|f\|\le a_1, \ \textup{ and }
    \|f\|_{-1}\le a_2\bigg\},
\end{multline}  
  for some fixed positive constant $c_1$ and $a_i$ for $i=1,2.$ Here given $m,$ $M$, and $C_l$ from \eqref{boundary condition} and Proposition \ref{lowerL},
  \begin{itemize}
      \item $c_1$ is chosen such that $c_1=\frac{cm}{2}$ for the uniform constant $c$ given in \eqref{LAf1 lower}.
      \item $a_1$ is chosen such that $a_1\ge 2M+ 2\frac{1}{C_l}$.
      \item $a_2$ is chosen such that $a_2\ge  2M+2\frac{a_1^2}{C_l}$.
    \item Lastly, $k$ is chosen sufficiently large such that $k\ge a_1a_2$ and that  \eqref{k sufficiently large} holds. Also, we require $k$ is sufficiently large such that $\frac{e^{-\frac{k}{\sqrt{2}}}}{C_l^2}(C_lM+a_1a_2)$ are uniformly bounded and that it can be shown in the proof of Lemma \ref{contraction lemma} that $A$ is a contraction mapping from $\mathcal{A}$ to itself. 
  \end{itemize}
 \subsection{Solution operator $A$}Recall that we obtain the mild representations of the stationary solution in \eqref{f representation for p1}. Define the right-hand side of \eqref{f representation for p1} as the operator $A$. More precisely, we define the solution operator $A$ as follows:
 \begin{equation}\begin{split}
     \label{def.A}
     Af(x_1,p)&\eqdef 1_{\{p_1>0\}}f_L(p)\exp\left(-\frac{1}{\hat{p}_1}\int_{0}^{x_1} \mathcal{L}f(z,p)dz\right)\\*&\quad+1_{\{p_1>0\}}\frac{1}{\hat{p}_1}\int_0^{x_1}\exp\left(-\frac{1}{\hat{p}_1}\int_z^{x_1} \mathcal{L}f(z',p)dz'\right)Q^+(f,f)(z,p)dz \\*
     &\quad+1_{\{p_1<0\}}f_R(p)\exp\left(-\frac{1}{\hat{p}_1}\int_1^{x_1} \mathcal{L}f(z,p)dz\right)\\*
&\quad+1_{\{p_1<0\}}\frac{1}{\hat{p}_1}\int_1^{x_1}\exp\left(-\frac{1}{\hat{p}_1}\int_z^{x_1} \mathcal{L}f(z',p)dz'\right)Q^+(f,f)(z,p)dz\\*
&=:(Af)_1+(Af)_2+(Af)_3+(Af)_4.
 \end{split}\end{equation}Note that for $x_1\in [0,1],$ each $(Af)_i$ for $i=1,2,3,4$ is non-negative given that $f$ is non-negative.
 Then we can first prove that $A$ is a self-map from $\mathcal{A}$ to itself as follows.

\begin{lemma}[Self-map property of $A$]\label{selfmap lemma}
For a sufficiently large $k>0,$ we have 
    $$A\mathcal{A}\subset \mathcal{A}.$$
\end{lemma}
\begin{proof}We check that if $f\in \mathcal{A}$ then $Af\in\mathcal{A}.$ Suppose $f \in \mathcal{A}.$ 
Note that we easily have $Af\ge 0$ since each term in \eqref{def.A} is non-negative. Also, we have 
\begin{equation}
    \notag
        (Af)(x_1,p)\ge 1_{p_1>0}f_L(p)\exp\left(-\frac{c_1 x_1}{\hat{p}_1}\right)
        +1_{p_1<0}f_R(p)\exp\left(\frac{(1-x_1)c_1}{\hat{p}_1}\right),
    \end{equation}using the definition of $A$ in \eqref{def.A} and the fact that $\mathcal{L}f\ge c_1,$ and that $Q^+(f,f)$ is non-negative. 
For the rest of the required estimates, we only check the first two terms $(Af)_1$ and $(Af)_2$ of \eqref{def.A} since the other terms, $(Af)_3$ and $(Af)_4$, can be treated similarly if we make a change of variables $(x_1,p)\mapsto (\bar{x}_1,\bar{p})\eqdef (1-x_1,-p)$.

\vspace{5mm}

\noindent\textbf{1. Proof for $\mathcal{L}(Af) \ge c_1$.} We show that $c_1$ can be chosen such that $\mathcal{L}((Af)_i)\ge \frac{c_1}{2},$ for $i=1,3$. We recall the definition of the set $Q_{r_1,r_2}\in \rth$ introduced in Section \ref{sec.boundary} for some $0<r_1<r_2<\infty$:
$$Q_{r_1,r_2}\eqdef \{q\in\rth: q_1\ge r_1\textup{ and } r_1\le |q|\le r_2\}. $$ Then note that 
\begin{align*}
    \mathcal{L}((Af)_1)(x_1,p)&= \int_\rth\int_{\mathbb{S}^2}v_{\textup{\o}} \sigma(g,\theta)(Af)_1(x_1,q)d\omega dq\\
    &=\int_\rth dq\int_{\mathbb{S}^2}d\omega\ v_{\textup{\o}} \sigma(g,\theta)1_{\{q_1>0\}}f_L(q)\exp\left(-\frac{1}{\hat{q}_1}\int_{0}^{x_1} \mathcal{L}f(z,q)dz\right)\\
    &=\int_\rth dq\int_{\mathbb{S}^2}d\omega\ v_{\textup{\o}} \sigma(g(p^\mu,q^\mu),\theta)1_{\{q_1>0\}}f_L(q)\\&\qquad\times \exp\left(-\frac{1}{\hat{q}_1}\int_{0}^{x_1} \int_\rth\int_{\mathbb{S}^2}v_{\textup{\o}} \sigma(g(q^\mu,\bar{q}^\mu),\theta)f(z,\bar{q})d\omega  d\bar{q}dz\right).
    \end{align*}
    Note that if we define $\bar{q}^0=\sqrt{1+|\bar{q}|^2},$ we observe that
    \begin{align*}
&\int_\rth\int_{\mathbb{S}^2}v_{\textup{\o}} \sigma(g(q^\mu,\bar{q}^\mu),\theta)f(z,\bar{q})d\omega  d\bar{q}
        \lesssim  c_0\int_\rth |q-\bar{q}|f(z,\bar{q})d\bar{q}\\*
          &\lesssim  c_0\int_\rth (|q|+|\bar{q}|)e^{-k\bar{q}^0}e^{k\bar{q}^0}f(z,\bar{q})d\bar{q}
          \lesssim  c_0\int_\rth \frac{|q|+|\bar{q}|}{1+k\bar{q}^0}\varphi(\bar{q})f(z,\bar{q})d\bar{q}\\*
           &\lesssim  c_0k^{-1} (1+r_2)\int_\rth \varphi(\bar{q})f(z,\bar{q})d\bar{q}\approx c_0k^{-1} (1+r_2)\|f\|
      \lesssim c_0k^{-1}(1+r_2)a_1,
    \end{align*}for $q\in Q_{r_1,r_2}$ using $v_{\textup{\o}} g(q^\mu,\bar{q}^\mu)\lesssim g(q^\mu,\bar{q}^\mu)\lesssim |q-\bar{q}|$ and then using $|q-\bar{q}|\le |q|+|\bar{q}|$ and $e^{-z}\lesssim \frac{1}{1+z}$. 
    Therefore, we have
    \begin{align*}
   & \mathcal{L}((Af)_1)(x_1,p) \\*
   & \gtrsim \int_{Q_{r_1,r_2}} dq\int_{\mathbb{S}^2}d\omega\ v_{\textup{\o}} \sigma(g,\theta)f_L(q)\exp\left(-\frac{1}{\hat{q}_1}\int_0^{x_1}c_0k^{-1}(1+r_2)a_1 dz\right)\\*
   & \gtrsim e^{-c_0k^{-1}r_1^{-1}(1+r_2)^2a_1 }\int_{Q_{r_1,r_2}} dq\int_{\mathbb{S}^2}d\omega\ v_{\textup{\o}} \sigma(g,\theta)f_L(q)\\
    & \approx e^{-c_0k^{-1}r_1^{-1}(1+r_2)^2a_1 }\int_{Q_{r_1,r_2}} dq\  c_0\frac{g^2\sqrt{s}}{p^0q^0}f_L(q)\\
     & \gtrsim e^{-c_0k^{-1}r_1^{-1}(1+r_2)^2a_1 }\int_{Q_{r_1,r_2}} dq\  c_0\frac{|p-q|^2}{(p^0q^0)^2}f_L(q),
    \end{align*}using Lemma \ref{gINEQ}, $s\ge 4$, $\hat{q}_1^{-1}\lesssim r_1^{-1}(1+r_2)$ on $Q_{r_1,r_2}$, and $x_1\in [0,1]$. By Lemma \ref{lemma m}, we have
    $$
    \mathcal{L}((Af)_1)(x_1,p) \gtrsim e^{-c_0k^{-1}r_1^{-1}(1+r_2)^2a_1 } c_0 m.$$
    Then we choose $k$ sufficiently large such that 
    \begin{equation}\label{k sufficiently large}
        e^{-c_0k^{-1}r_1^{-1}(1+r_2)^2a_1 } \ge \frac{1}{4}.
    \end{equation} Then we can conclude that for some uniform constant $c>0,$ we have 
    \begin{equation}
        \label{LAf1 lower}
 \mathcal{L}((Af)_1)(x_1,p)\ge \frac{cm}{4},    \end{equation} for any $x_1,p$. Similarly, one can show that $\mathcal{L}((Af)_3)(x_1,p)\ge \frac{cm}{4}. $ If $c_1$ has been chosen such that $c_1 = \frac{cm}{2}$ for the given $m$ in the boundary condition \eqref{boundary condition}, we can here conclude the proof for $\mathcal{L}(Af)(x_1,p)\ge c_1.$

\vspace{5mm}

\noindent\textbf{2. Proof for $\|Af\|\le a_1.$}
Note that we have $\mathcal{L}f\ge C_l (p^0)^{1/2}$ via Proposition \ref{lowerL} and hence 
\begin{align*}
    \notag
&\|(Af)_1\|=\int_\rth dp\ (p^0)^{1/2}\varphi(p) 1_{\{p_1>0\}}f_L(p)\exp\left(-\frac{1}{\hat{p}_1}\int_{0}^{x_1} \mathcal{L}f(z,p)dz\right)\\
&\le \int_\rth dp\ (p^0)^{1/2}\varphi(p)1_{\{p_1>0\}}f_L(p)\exp\left(-\frac{1}{\hat{p}_1}C_l(p^0)^{1/2}x_1\right)\le \|f_{LR}\|\le M.
\end{align*}
On the other hand, note that we have
\begin{align*}
    \notag
&\|(Af)_2\|=\int_\rth dp\ (p^0)^{1/2}\varphi(p) 1_{\{p_1>0\}}\frac{1}{\hat{p}_1}\\*&\qquad\times\int_0^{x_1}\exp\left(-\frac{1}{\hat{p}_1}\int_z^{x_1} \mathcal{L}f(z',p)dz'\right) Q^+(f,f)(z,p)dz\\
&\le \int_\rth dp\ (p^0)^{1/2}\varphi(p)1_{\{p_1>0\}}\frac{1}{\hat{p}_1}\int_0^{x_1}\exp\left(-\frac{1}{\hat{p}_1}C_l(p^0)^{1/2} (x_1-z)\right) Q^+(f,f)(z,p)dz\\
&\le \int_\rth dp\ (p^0)^{1/2}\varphi(p)1_{\{p_1>0\}}\frac{1}{\hat{p}_1}\sup_{z\in [0,1]}|Q^+(f,f)(z,p)|\\*&\qquad\times\int_0^{x_1}\exp\left(-\frac{1}{\hat{p}_1}C_l(p^0)^{1/2} (x_1-z)\right) dz\\
&\le \int_\rth dp\ (p^0)^{1/2}\varphi(p)1_{\{p_1>0\}}\sup_{z\in [0,1]}|Q^+(f,f)(z,p)|\frac{1}{C_l (p^0)^{1/2}}\\*&\qquad\times \left(1-\exp\left(-\frac{1}{\hat{p}_1}C_l(p^0)^{1/2} x_1\right) \right)\\
&\le\frac{1}{C_l } \int_\rth dp\ \varphi(p)\sup_{z\in [0,1]}|Q^+(f,f)(z,p)|
\le \frac{1}{kC_l}\|f\|_{-1} \|f\|\le  \frac{a_1a_2}{kC_l},
\end{align*}using Lemma \ref{lemma.pointwise.gain of weight}, that $\|f\|\le a_1$, $\|f\|_{-1}\le a_2,$ and that $\mathcal{L}f\ge C_l (p^0)^{1/2}$ via Proposition \ref{lowerL}. Let $k$ be sufficiently large such that $k\ge a_1a_2.$ 

Therefore, given $M$ and $C_l$, if $a_1$ has been chosen such that $a_1\ge 2M+ 2\frac{1}{C_l}$ for a sufficiently large $k$, we obtain the upper bound that $\|Af\|\le a_1$. In addition, as a direct consequence, we also have 
$$Af\in  L^\infty_{x_1}\left([0,1];L^1_p\!\left(\mathbb{R}^3,\,(p^0)^{1/2}\,dp\right)\right),$$ by the inequalities \eqref{function space ineq}.

\vspace{5mm}

\noindent \textbf{3. Proof for $\|Af\|_{-1}\le a_2.$} Note that we have $\mathcal{L}f\ge C_l (p^0)^{1/2}$ via Proposition \ref{lowerL} and hence
\begin{align*}
    \notag
&\|(Af)_1\|_{-1}=\sup_{a\in \rth}\int_\rth dp\ \frac{(p^0)^{1/2}\varphi(p) }{|p-a|}1_{\{p_1>0\}}f_L(p)\exp\left(-\frac{1}{\hat{p}_1}\int_{0}^{x_1} \mathcal{L}f(z,p)dz\right)\\*
&\le \sup_{a\in \rth}\int_\rth dp\ \frac{(p^0)^{1/2}\varphi(p) }{|p-a|}1_{\{p_1>0\}}f_L(p)\exp\left(-\frac{1}{\hat{p}_1}\int_{0}^{x_1} C_l(p^0)^{1/2}dz\right)\\*
&\le \sup_{a\in \rth}\int_\rth dp\ \frac{(p^0)^{1/2}\varphi(p) }{|p-a|}1_{\{p_1>0\}}f_L(p)\exp\left(-\frac{1}{\hat{p}_1} C_l(p^0)^{1/2}x_1\right)
\le  \|f_{LR}\|_{-1}\le M.
\end{align*}
On the other hand, note that we have
\begin{align*}    \notag
&\|(Af)_2\|_{-1}=\sup_{a\in \rth}\int_\rth dp\ \frac{(p^0)^{1/2}\varphi(p) }{|p-a|}1_{\{p_1>0\}}\frac{1}{\hat{p}_1}\int_0^{x_1}\exp\left(-\frac{1}{\hat{p}_1}\int_z^{x_1} \mathcal{L}f(z',p)dz'\right)\\&\qquad\qquad\times Q^+(f,f)(z,p)dz\\&
\le \sup_{a\in \rth}\int_\rth dp\ \frac{(p^0)^{1/2}\varphi(p) }{|p-a|}1_{\{p_1>0\}}\frac{1}{\hat{p}_1}\int_0^{x_1}\exp\left(-\frac{1}{\hat{p}_1}C_l(p^0)^{1/2} (x_1-z)\right) \\&\qquad\qquad\times Q^+(f,f)(z,p)dz\\
&\le \sup_{a\in \rth}\int_\rth dp\ \frac{(p^0)^{1/2}\varphi(p) }{|p-a|}1_{\{p_1>0\}}\frac{1}{C_l (p^0)^{1/2}}\left(1-\exp\left(-\frac{1}{\hat{p}_1}C_l(p^0)^{1/2} x_1\right) \right)\\&\qquad\qquad\times\|Q^+(f,f)(\cdot,p)\|_{L^\infty_{x_1}}\\
&\le \frac{1}{C_l}\sup_{a\in \rth}\int_\rth dp\ \frac{\varphi(p) }{|p-a|} \|Q^+(f,f)(\cdot,p)\|_{L^\infty_{x_1}}
\le \frac{1}{C_l}\|f\|^2 \le  \frac{a_1^2}{C_l},
\end{align*}using Proposition \ref{lemma1}, that $\|f\|\le a_1$, and that $\mathcal{L}f\ge C_l (p^0)^{1/2}$ via Proposition \ref{lowerL}. 

Altogether, given $a_1$, $M$ and $C_l$, if $a_2$ has been chosen such that  $a_2\ge  2M+2\frac{a_1^2}{C_l},$ we obtain that $\|Af\|_{-1}\le a_2$.

\end{proof}
\section{Existence and Uniqueness of Steady Solutions}
\label{sec.5}

\subsection{Contraction of the solution operator $A$}
 
 In this section, we prove that $A$ is continuous in the following sense:
\begin{lemma}[Contraction property of $A$]\label{contraction lemma}Suppose $f,h\in\mathcal{A}$. $f_{LR}$ further satisfies the conditions in \eqref{boundary condition}. Suppose $f$ and $h$ coincide at the incoming boundary as $f|_{\gamma_-}=h|_{\gamma_-}=f_{LR}|_{\gamma_-}$.  Then we have 
$$\|Af-Ah\|\lesssim k^{-1} \|f-h\|, $$for $k>0$ where the norm $\|\cdot\|$ is defined as 
\begin{equation}
    \label{norm.||}
    \|f\|\eqdef \int_\rth (p^0)^{1/2}\varphi(p)\|f(\cdot,p)\|_{L^\infty_{x_1}}dp.
\end{equation}
\end{lemma}
\begin{proof}
    Since $f$ and $h$ coincide on the incoming boundary $\gamma_-,$ the differences $(Af)_1-(Ah)_1$ and $(Af)_3-(Ah)_3$ of \eqref{def.A} are cancelled out in the $Af-Ah$ representation. We now look at the difference $(Af)_2-(Ah)_2$ of \eqref{def.A}. The other difference $(Af)_4-(Ah)_4$ is similar and we omit the estimate. 
    
    \noindent \textbf{1. Estimate for the difference $(Af)_2-(Ah)_2$. } Note that
    \begin{align*}
     &   (Af)_2(x_1,p)-(Ah)_2(x_1,p)\\
       & =1_{\{p_1>0\}}\frac{1}{\hat{p}_1}\int_0^{x_1}\exp\left(-\frac{1}{\hat{p}_1}\int_z^{x_1} \mathcal{L}f(z',p)dz'\right)\bigg[Q^+(f,f)(z,p)-Q^+(h,h)(z,p)\bigg]dz\\
       & +1_{\{p_1>0\}}\frac{1}{\hat{p}_1}\int_0^{x_1} \left[e^{\left(-\frac{1}{\hat{p}_1}\int_z^{x_1} \mathcal{L}f(z',p)dz'\right)}-e^{\left(-\frac{1}{\hat{p}_1}\int_z^{x_1} \mathcal{L}h(z',p)dz'\right)}\right]Q^+(h,h)(z,p)dz.
    \end{align*}Note that by the mean-value theorem we have $|e^{-x}-e^{-y}|\le e^{-\min\{x,y\}} |x-y|$. Also by Lemma \ref{L.cont}, we have
     \begin{equation}\begin{split}\label{arg1}
     &  | (Af)_2(x_1,p)-(Ah)_2(x_1,p)|\\
      &  \le 1_{\{p_1>0\}}\frac{1}{\hat{p}_1}\int_0^{x_1}\exp\left(-\frac{1}{\hat{p}_1}\int_z^{x_1} \mathcal{L}f(z',p)dz'\right)\bigg|Q^+(f,f)(z,p)-Q^+(h,h)(z,p)\bigg|dz\\
      &  +1_{\{p_1>0\}}\frac{1}{\hat{p}^2_1}\int_0^{x_1}dz\ e^{\left(-\frac{1}{\hat{p}_1}\min\{\int_z^{x_1} \mathcal{L}f(z',p)dz',\int_z^{x_1} \mathcal{L}h(z',p)dz'\}\right)}\\&\qquad\qquad\times e^{-\frac{k}{\sqrt{2}}}(p^0)^{1/2}\|f-h\|(x_1-z)Q^+(h,h)(z,p)\\
        &\eqdef I_1(x_1,p)+I_2(x_1,p).
    \end{split}\end{equation}

    \noindent\textbf{(i) Estimate for $I_2$.} 
    Since $\mathcal{L}f\ge C_l(p^0)^{1/2}$ and $\mathcal{L}h\ge C_l(p^0)^{1/2}$ by Proposition \ref{lowerL}, we have 
    \begin{equation}
        \label{arg2}-\min\left\{\int_z^{x_1} \mathcal{L}f(z',p)dz',\int_z^{x_1} \mathcal{L}h(z',p)dz'\right\}\le -C_l(p^0)^{1/2}(x_1-z).
    \end{equation} Thus, using $0\le z\le x_1\le 1,$ we have \begin{align*}
     \varphi(p)  I_2(x_1,p)&=  \varphi(p)1_{\{p_1>0\}}\frac{1}{\hat{p}^2_1}\int_0^{x_1}dz\ \exp\left(-\frac{C_l(p^0)^{1/2}}{\hat{p}_1}(x_1-z)\right)\\*&\qquad\qquad\times e^{-\frac{k}{\sqrt{2}}}(p^0)^{1/2}\|f-h\|(x_1-z)Q^+(h,h)(z,p)\\
   & \le \varphi(p)1_{\{p_1>0\}}\frac{1}{\hat{p}^2_1}(\sup_z |Q^+(h,h)(z,p)|)\int_0^\infty \frac{\hat{p}_1 dz'}{C_l(p^0)^{1/2}}\ \exp\left(-z'\right)\\&\qquad\qquad\times e^{-\frac{k}{\sqrt{2}}}(p^0)^{1/2}\|f-h\|\frac{\hat{p}_1}{C_l(p^0)^{1/2}}z'\\
       & \lesssim \varphi(p) 1_{\{p_1>0\}}\sup_{z}|Q^+(h,h)(z,p)|\frac{e^{-\frac{k}{\sqrt{2}}}}{C_l^2}(p^0)^{-1/2}\|f-h\|,
    \end{align*}where we made a change of variables $z\mapsto z'\eqdef -\frac{C_l(p^0)^{1/2}}{\hat{p}_1}(x_1-z).$ 
    Thus, \begin{align*}
        \|I_2\|&\lesssim \frac{e^{-\frac{k}{\sqrt{2}}}}{C_l^2}\|f-h\|\int_\rth dp\  \varphi(p) 1_{\{p_1>0\}}\sup_{z}|Q^+(h,h)(z,p)|\\
        &\lesssim \frac{e^{-\frac{k}{\sqrt{2}}}}{kC_l^2}\|f-h\|\|h\|_{-1} \|h\|
        \lesssim \frac{e^{-\frac{k}{\sqrt{2}}}}{k} \frac{a_1a_2}{C_l^2}\|f-h\|,
    \end{align*}using that $h\in \mathcal{A}$ and Lemma \ref{lemma.pointwise.gain of weight}.

\noindent\textbf{(ii) Estimate for $I_1$.} 
We now look at the difference of $Q^+(f,f)-Q^+(h,h)$ appearing in the difference $I_1$. Since $Q^+$ is bi-linear, we first observe that for each $p$,
$$\|Q^+(f,f)-Q^+(h,h)\|_{L^\infty_{x_1}}\le \frac{1}{2}\left(\|Q^+(f-h,f+h)\|_{L^\infty_{x_1}}+\|Q^+(f+h,f-h)\|_{L^\infty_{x_1}}\right).$$ Therefore, using $\mathcal{L}f\ge C_l(p^0)^{1/2}$ from Proposition \ref{lowerL}, we note that 
\begin{align*}
 \|I_1\|&=  \sup_{x_1} \int_\rth dp\ (p^0)^{1/2}\varphi(p) |I_1(x_1,p)| \\
  & \le  \sup_{x_1} \int_\rth dp\ (p^0)^{1/2}\varphi(p)1_{\{p_1>0\}}\frac{1}{2\hat{p}_1}\int_0^{x_1}\exp\left(-\frac{1}{\hat{p}_1}C_l(p^0)^{1/2}(x_1-z)\right)\\&\qquad\qquad\times \left(\|Q^+(f-h,f+h)\|_{L^\infty_{x_1}}+\|Q^+(f+h,f-h)\|_{L^\infty_{x_1}}\right)dz\\
 &  =\sup_{x_1} \int_\rth dp\ \varphi(p)1_{\{p_1>0\}}\frac{1}{2C_l}\left(1-e^{-\frac{1}{\hat{p}_1}C_l(p^0)^{1/2}x_1}\right) \\&\qquad\qquad\times \left(\|Q^+(f-h,f+h)\|_{L^\infty_{x_1}}+\|Q^+(f+h,f-h)\|_{L^\infty_{x_1}}\right)\\&\eqdef J_1+J_2.
\end{align*}
By Corollary \ref{cor.pointwise.gain of weight}, we have$$J_1\lesssim \frac{1}{2kC_l}\|(f+h)\|_{-1}\|f-h\|.$$
By Lemma \ref{lemma.pointwise.gain of weight}, we have
$$J_2\lesssim \frac{1}{2kC_l}\|(f+h)\|_{-1}\|f-h\|.$$Thus, we have 
$$\|I_1\|\lesssim \frac{2a_2}{kC_l}\|f-h\|,$$ under the assumption that $f,h\in \mathcal{A}$ so that  $\|f+h\|_{-1}\le 2a_2$.

\noindent\textbf{2. Estimate for the difference $(Af)_1-(Ah)_1$. } Since $f$ and $h$ have the same boundary profile $f|_{\gamma_-}=h|_{\gamma_-}=f_{LR}|_{\gamma_-}$, the difference $(Af)_1-(Ah)_1$ is written as
\begin{multline*}
    (Af)_1(x_1,p)-(Ah)_1(x_1,p)\\*=1_{\{p_1>0\}}f_L(p)\left(\exp\left(-\frac{1}{\hat{p}_1}\int_{0}^{x_1} \mathcal{L}f(z,p)dz\right)-\exp\left(-\frac{1}{\hat{p}_1}\int_{0}^{x_1} \mathcal{L}h(z,p)dz\right)\right).
\end{multline*}Using Proposition \ref{lowerL}, Lemma \ref{L.cont}, and the same argument of \eqref{arg1}--\eqref{arg2}, we have
\begin{align*}
    &|(Af)_1(x_1,p)-(Ah)_1(x_1,p)|\\*&\le 1_{\{p_1>0\}}f_L(p)\frac{1}{\hat{p}_1}e^{\left(-\frac{1}{\hat{p}_1}\min\{\int_{0}^{x_1} \mathcal{L}f(z,p)dz,\int_{0}^{x_1} \mathcal{L}h(z,p)dz\}\right)}e^{-\frac{k}{\sqrt{2}}}(p^0)^{1/2}\|f-h\|x_1\\
    &\le 1_{\{p_1>0\}}f_L(p)\frac{1}{\hat{p}_1}e^{\left(-\frac{1}{\hat{p}_1}C_l(p^0)^{1/2}x_1\right)}e^{-\frac{k}{\sqrt{2}}}(p^0)^{1/2}\|f-h\|x_1\\
    &\lesssim  1_{\{p_1>0\}}f_L(p)\frac{1}{C_l}e^{\left(-(1-\epsilon)\frac{1}{\hat{p}_1}C_l(p^0)^{1/2}x_1\right)}e^{-\frac{k}{\sqrt{2}}}\|f-h\|,
\end{align*}using $ze^{-z} \lesssim e^{-(1-\epsilon)z}$ for any sufficiently small $\epsilon>0$ with $z=\frac{C_l(p^0)^{1/2}}{\hat{p}_1}x_1$ in our case. Therefore, we conclude that
\begin{align*}
   & \|(Af)_1-(Ah)_1\|\\*
  &  \lesssim e^{-\frac{k}{\sqrt{2}}}\|f-h\|\int_\rth dp\ (p^0)^{1/2}\varphi(p)1_{\{p_1>0\}}f_L(p)\frac{1}{C_l}\sup_{x_1} e^{\left(-(1-\epsilon)\frac{1}{\hat{p}_1}C_l(p^0)^{1/2}x_1\right)}\\*
    &\lesssim \frac{e^{-\frac{k}{\sqrt{2}}}}{C_l}\|f_{LR}\|\|f-h\|\lesssim \frac{Me^{-\frac{k}{\sqrt{2}}}}{C_l}\|f-h\|.
\end{align*}
Therefore for a sufficiently large $k>0$, the map $A$ is a contraction from $\mathcal{A}$ to itself supported by Lemma \ref{selfmap lemma} on the self-map property of $A$. This completes the proof.
 \end{proof}

\subsection{Existence and uniqueness of a stationary solution}
Using Lemma~\ref{selfmap lemma} and Lemma~\ref{contraction lemma}, we now prove
Theorem~\ref{thm:existence}.

\begin{proof}[Proof of Theorem \ref{thm:existence}] By Lemma \ref{selfmap lemma} and Lemma \ref{contraction lemma} for a sufficiently large $k>0$, we obtain that $A$ is a contraction mapping from $\mathcal{A}$ to itself. Thus, by the Banach fixed-point theorem, we obtain a unique fixed-point $f\in\mathcal{A}$ of the contraction mapping $A$ such that $f=Af$. This completes the proof of the existence of a unique solution to the stationary boundary-value problem. Note that the weighted $L^1_pL^\infty_{x_1}$ norm-boundedness and the $(-\Delta_p)^{-1}$ regularity in Theorem \ref{thm:existence} also follow by the definition of the solution space $\mathcal{A}$ in \eqref{sol space}.
\end{proof}
In addition, assuming that the boundary profile satisfies a hyperplane estimate $\|f_{LR}\|_{\textup{hyp}}<+\infty$, we can further establish that the corresponding solution satisfies the same estimate.
 The detailed proof is presented in the following section.

\section{Propagation of Hyperplane Estimates}
\label{sec.6}

In this section, we prove that if the inflow boundary profile satisfies the
hyperplane bound $\|f_{LR}\|_{\mathrm{hyp}}<\infty$, then the corresponding
steady solution $f$ also satisfies $\|f\|_{\mathrm{hyp}}<\infty$.
The proof relies crucially on the hyperplane estimate for the gain term
$Q^+(f,f)$ established in Proposition~\ref{prop.3.4} below, together with the mild formulation
of the stationary problem.

We emphasize that, unlike
\cite{MR459450,MR2122556,2503.12587}, this hyperplane estimate is not used in the
existence and uniqueness theory developed in Section~5, but is obtained here as
an \emph{a posteriori} regularity property of the steady solution.

\subsection{Hyperplane integral of $Q^+$} We now estimate the integral of the gain term $Q^+$ on a hyperplane.
\begin{proposition}\label{prop.3.4}
Let $\varphi(p) \eqdef e^{k p^0}$ for some $k>0$.  
Then we have
\begin{equation}\label{hyperplane}
\sup_{E}
\int_{E}
(p^0)^{1/2}\,\varphi(p)\,
\|Q^+(f,f)(\cdot,p)\|_{L^\infty_{x_1}}
\, d\sigma_E(p)
\;\lesssim\;
\|f\|^2,
\end{equation}
where the supremum is taken over all two-dimensional planes $E \subset \mathbb{R}^3$ passing through the origin, and $d\sigma_E$ denotes the two-dimensional Lebesgue measure on $E$.
\end{proposition}

\begin{proof}
Let $a_E$ be a 3-dimensional unit vector in $\rth$ that is perpendicular to the plane $E$.
We first write
\begin{align*}
  K_E&\eqdef   \int_E d\sigma_E(p) \hspace{1mm}(p^0)^{1/2}\varphi(p)Q^+(f,f)(x_1,p) \\
  &= \int_\rth dp \hspace{1mm}(p^0)^{1/2}\varphi(p)Q^+(f,f)(x_1,p) \delta( p\cdot a_E)\\&=\int_\rth \frac{dp}{p^0}\int_\rth\frac{dq}{q^0}\int_{\mathbb{S}^2}d\omega\  (p^0)^{1/2}g\sqrt{s} g  \sigma_0(\theta)\varphi(p)f(x_1,p^{\prime})f(x_1,q^{\prime})\delta(p\cdot a_E),
\end{align*}using the definition $v_{\textup{\o}} = \frac{g\sqrt{s}}{p^0q^0}.$ Define a 4-vector $a^\mu\eqdef (0,a_E)$. Then we can write $p\cdot a_E = p^\mu a_\mu. $ Now we consider the momentum rotation within the Lorentz transformation as in \eqref{rot.lorentz} such that $\Lambda a^\mu = (0,0,0,1)^\top.$ Then by the invariance property \eqref{invariance}, we have 
$$p^\mu a_\mu=(\Lambda^{\kappa}_{~\mu}  p^{\mu})\eta_{\kappa\lambda } (\Lambda^{\lambda}_{~\nu}  a^{\nu} )= (\Lambda p^\mu)(\Lambda a_\mu),$$ by abusing the notation and writing $\Lambda p^\kappa = \Lambda^{\kappa}_{~\mu}  p^{\mu}$. Similarly, we have 
$$g^2=g^2(p^\mu,q^\mu)= (p^\mu-q^\mu)(p_\mu-q_\mu)= g^2(\Lambda p^\mu, \Lambda q^\mu),$$ and $$s=s(p^\mu,q^\mu)= -(p^\mu+q^\mu)(p_\mu+q_\mu)= s(\Lambda p^\mu, \Lambda q^\mu).$$ Note that the measure $\frac{dp}{p^0}\frac{dq}{q^0}$ is invariant under this transformation. Lastly, $p^0$ and $q^0$ (and hence $\varphi(p)$) are also invariant under this momentum rotation. Then by denoting 
the unit vector $(0,0,1)^\top$ as 
 $\hat{e}_3$, we have
\begin{equation}\notag
  K_E  =\int_\rth \frac{dp}{p^0}\int_\rth\frac{dq}{q^0}\int_{\mathbb{S}^2}d\omega\ (p^0)^{1/2}g\sqrt{s} g  \sigma_0(\theta)\varphi(p)f(x_1,p^{\prime})f(x_1,q^{\prime})\delta(p\cdot\hat{e}_3),
\end{equation} by relabeling $\Lambda p^\mu$ and $\Lambda q^\mu$ as $p^\mu$ and $q^\mu$.  
Now we apply a pre-post collisional change of the variables $(p^{\mu}, q^{\mu}) \to (p^{\prime\mu}, q^{\prime\mu})$ and use the fact that  $s$ and $g$ are invariant under this transformation. Then we obtain that the integral $K_E$ is now given by
\begin{align}\notag
K_E&\lesssim  \int_\rth \int_\rth\int_{\mathbb{S}^2} (p'^0)^{1/2}v_{\textup{\o}}g  \varphi(p')f(p)f(q)\delta(p'_3)d\omega dqdp,
\end{align} where we used \eqref{kernel condition} and that $\sigma_0(\theta)\le M$ and denoted $f(x_1,p)$ as $f(p)$ for the sake of notational simplicity. Since $(p')^0\le p^0+q^0$ by the energy conservation law \eqref{collision.invariants} (and hence $(p')^0\le 2 p^0q^0$), we have 
$\varphi(p')\le \varphi(p)\varphi(q),$ and hence\begin{align}\label{K}
K_E&\lesssim  \int_\rth \int_\rth\int_{\mathbb{S}^2} (p^0)^{1/2}(q^0)^{1/2}v_{\textup{\o}}g  \varphi(p)\varphi(q)f(p)f(q)\delta(p'_3)d\omega dqdp.
\end{align} 
Now by \eqref{p'}, we observe that
\begin{equation*}
    \int_{\mathbb{S}^2}\delta(p'_3)d\omega=\int_{\mathbb{S}^2}\delta\left(\frac{p_3+q_3}{2}+\frac{g}{2}\left(\omega_3+(\gamma-1)(p_3+q_3)\frac{(p+q)\cdot\omega}{|p+q|^2}\right)\right)d\omega.
\end{equation*}
Note that 
$$\omega_3+(\gamma-1)(p_3+q_3)\frac{(p+q)\cdot\omega}{|p+q|^2} = \left((0,0,1)+(\gamma-1)(p_3+q_3)\frac{(p+q)}{|p+q|^2}\right)\cdot \omega .$$ Then we consider another representation for $\omega$ using the spherical coordinates such that $z$-axis is especially chosen to be parallel to the vector $$Z\eqdef \left((0,0,1)+(\gamma-1)(p_3+q_3)\frac{(p+q)}{|p+q|^2}\right)$$ such that $\omega = (\sin\theta\cos\phi,\sin\theta\sin\phi,\cos\theta)$ and$$\cos\theta = \frac{Z}{|Z|}\cdot \omega.$$ Then the integral can be rewritten as
\begin{multline*}
    \int_{\mathbb{S}^2}\delta(p'_3)d\omega=\int_0^{2\pi}d\phi \int_0^\pi d\theta\  \sin\theta\ \delta\left(\frac{p_3+q_3}{2}+\frac{g}{2}|Z|\cos\theta\right)\\*
 =\int_0^{2\pi}d\phi \int_{-1}^1 dk\   \delta\left(\frac{p_3+q_3}{2}+\frac{g}{2}|Z|k\right)
=  \frac{4\pi}{g|Z|}\int_{-1}^1 dk\   \delta\left(k+ \frac{p_3+q_3}{g|Z|}\right), 
\end{multline*}where we made the change of variables $\theta\mapsto k\eqdef \cos\theta,$ and used $$\delta(Ak+B)=\frac{1}{|A|}\delta\left(k+\frac{B}{A}\right).$$
Note that 
\begin{equation}\begin{split}\label{one over Z estimate}
\frac{|p_3+q_3|}{g|Z|}&=\frac{|p_3+q_3|}{g\left|(0,0,1)+(\gamma-1)(p_3+q_3)\frac{(p+q)}{|p+q|^2}\right|}
=\frac{|p_3+q_3|}{g\left|(0,0,1)+\frac{(p_3+q_3)(p+q)}{\sqrt{s}(\pZ+\qZ+\sqrt{s})}\right|}\\
&=\frac{\sqrt{s}(\pZ+\qZ+\sqrt{s})|p_3+q_3|}{g\left|(0,0,\sqrt{s}(\pZ+\qZ+\sqrt{s}))+(p_3+q_3)(p+q)\right|}\\
&=\frac{\sqrt{s}(\pZ+\qZ+\sqrt{s})|p_3+q_3|}{g\sqrt{(p_3+q_3)^2\left(\sum_{i=1}^2(p_i+q_i)^2\right)+[\sqrt{s}(\pZ+\qZ+\sqrt{s})+(p_3+q_3)^2]^2}}\\
&=\frac{\sqrt{s}(\pZ+\qZ+\sqrt{s})}{g\sqrt{(p_1+q_1)^2+(p_2+q_2)^2+[\frac{\sqrt{s}(\pZ+\qZ+\sqrt{s})}{|p_3+q_3|}+|p_3+q_3|]^2}},
\end{split}\end{equation}
where the second identity is by \eqref{gamma minus one}. Therefore, we have
$$\frac{1}{g|Z|}=\frac{\sqrt{s}(\pZ+\qZ+\sqrt{s})}{g\sqrt{(p_3+q_3)^2\left(\sum_{i=1}^2(p_i+q_i)^2\right)+[\sqrt{s}(\pZ+\qZ+\sqrt{s})+|p_3+q_3|^2]^2}}.$$Then by computing the delta function in the integral with respect to $k,$ we have \begin{align*}
    K_E&\lesssim\int_\rth \int_\rth (p^0)^{1/2}(q^0)^{1/2}v_{\textup{\o}}g  \varphi(p)\varphi(q)f(p)f(q)\frac{4\pi}{g|Z|} dqdp\\
&\approx  \int_\rth \int_\rth dqdp\ \frac{g  \sqrt{s}}{p^0q^0}\\*&\qquad\times \frac{(p^0)^{1/2}(q^0)^{1/2}\varphi(p)\varphi(q)f(p)f(q)\sqrt{s}(\pZ+\qZ+\sqrt{s})}{ \sqrt{(p_3+q_3)^2\left(\sum_{i=1}^2(p_i+q_i)^2\right)+[\sqrt{s}(\pZ+\qZ+\sqrt{s})+|p_3+q_3|^2]^2}}\\
&\lesssim  \int_\rth \int_\rth dqdp\ \frac{g  \sqrt{s}}{p^0q^0} \frac{(p^0)^{1/2}(q^0)^{1/2}\varphi(p)\varphi(q)f(p)f(q)\sqrt{s}(\pZ+\qZ+\sqrt{s})}{ \sqrt{s}(\pZ+\qZ+\sqrt{s})}\\*
&\lesssim \int_\rth \int_\rth dqdp\ (p^0)^{1/2}(q^0)^{1/2}\varphi(p)\varphi(q)f(p)f(q)\lesssim \| f\|^2,
\end{align*}by \eqref{gamma minus one}, \eqref{K}, \eqref{one over Z estimate}, and the fact that $\sqrt{s}(\pZ+\qZ+\sqrt{s})$ is positive and 
$g^2,s\lesssim p^0q^0$. By taking ess$\sup$ over $x_1\in[0,1]$ and $\sup$ over all planes $E$ passing through 0, we complete the proof of Proposition \ref{prop.3.4}.
\end{proof}
\begin{remark}
In the proof above, one may ask whether the bound
\[
\left|\frac{p_3 + q_3}{g\,|Z|}\right| \le 1
\]
always holds. In general, this quantity can become arbitrarily large as $p$ approaches $q$.  
Indeed, in this regime we have $g \to 0$, and the numerator $p_3 + q_3$ does not compensate for this singularity unless $p_3 + q_3 = 0$.  

In such cases, the corresponding delta-function integral in $K_E$ vanishes identically, while the stated upper bound remains valid.  
This situation corresponds precisely to the case $p_3' \neq 0$.
\end{remark}

\subsection{Propagation of hyperplane estimate}

Let $f \in \mathcal{A}$ be a stationary solution as defined in \eqref{sol space}. Suppose further that the boundary profile $f_{LR}$ satisfies the hyperplane bound
$$
\|f_{LR}\|_{\textup{hyp}} \le M
$$
for some constant $M > 0$. We now establish the propagation of hyperplane estimate (Theorem \ref{thm:hyperplane}). 

\begin{proof}[Proof of Theorem \ref{thm:hyperplane}]
We recall the mild form representation \eqref{f representation for p1} for $f$ and define 
\begin{equation}
    \begin{split}
         (f)_1(x_1,p)&\eqdef 1_{\{p_1>0\}}f_L(p)\exp\left(-\frac{1}{\hat{p}_1}\int_{0}^{x_1} \mathcal{L}f(z,p)dz\right)\\
       (f)_2(x_1,p)&\eqdef  1_{\{p_1>0\}}
\frac{1}{\hat{p}_1}\int_0^{x_1}\exp\left(-\frac{1}{\hat{p}_1}\int_z^{x_1} \mathcal{L}f(z',p)dz'\right)Q^+(f,f)(z,p)dz\\
(f)_3(x_1,p)&\eqdef 1_{\{p_1<0\}}f_R(p)\exp\left(\frac{1}{\hat{p}_1}\int_{x_1}^1 \mathcal{L}f(z,p)dz\right),\textup{ and }\\
(f)_4(x_1,p)&\eqdef-1_{\{p_1<0\}}\frac{1}{\hat{p}_1}\int_{x_1}^1\exp\left(\frac{1}{\hat{p}_1}\int_{x_1}^z\mathcal{L}f(z',p)dz'\right)Q^+(f,f)(z,p)dz.
\end{split}
\end{equation}
In the remainder of this section, we estimate the hyperplane norms of $(f)_1$
and $(f)_2$.
 The rest of the terms $(f)_3$ and $(f)_4$ follow similar estimates and we omit them.

  Note that we have $\mathcal{L}f\ge C_l (p^0)^{1/2}$ via Proposition \ref{lowerL} and hence \begin{align*}
    \notag
\|(f)_1\|_{\textup{hyp}}&=\sup_E\int_E d\sigma_E(p)\ (p^0)^{1/2}\varphi(p) 1_{\{p_1>0\}}f_{LR}(p)\exp\left(-\frac{1}{\hat{p}_1}\int_{0}^{x_1} \mathcal{L}f(z,p)dz\right)\\
&\le \sup_E\int_E d\sigma_E(p)\ (p^0)^{1/2}\varphi(p) 1_{\{p_1>0\}}f_{LR}(p)\exp\left(-\frac{1}{\hat{p}_1}\int_{0}^{x_1} C_l(p^0)^{1/2}dz\right)\\
&\le \sup_E\int_E d\sigma_E(p)\ (p^0)^{1/2}\varphi(p) 1_{\{p_1>0\}}f_{LR}(p)\exp\left(-\frac{1}{\hat{p}_1} C_l(p^0)^{1/2}x_1\right)\\
&\le  \|f_{LR}\|_{\textup{hyp}}
\le M.
\end{align*}

On the other hand, regarding $(f)_2$ term (and similarly $(f)_4$ term), note that we have
\begin{align*}
    \notag
\|(f)_2\|_{\textup{hyp}}&=\sup_E\int_E d\sigma_E(p)\ (p^0)^{1/2}\varphi(p) 1_{\{p_1>0\}}\frac{1}{\hat{p}_1}\\&\qquad\qquad\times \sup_{x_1}\int_0^{x_1}\exp\left(-\frac{1}{\hat{p}_1}\int_z^{x_1} \mathcal{L}f(z',p)dz'\right) Q^+(f,f)(z,p)dz\\
&\le \sup_E\int_E d\sigma_E(p)\ (p^0)^{1/2}\varphi(p)1_{\{p_1>0\}}\frac{1}{\hat{p}_1}\\&\qquad\qquad\times  \sup_{x_1}\int_0^{x_1}\exp\left(-\frac{1}{\hat{p}_1}C_l(p^0)^{1/2} (x_1-z)\right)Q^+(f,f)(z,p)dz\\
&\le \sup_E\int_E d\sigma_E(p)\ (p^0)^{1/2}\varphi(p)1_{\{p_1>0\}}\frac{1}{\hat{p}_1}\|Q^+(f,f)(\cdot,p)\|_{L^\infty_{x_1}}\\*&\qquad\qquad\times \sup_{x_1}\int_0^{x_1}\exp\left(-\frac{1}{\hat{p}_1}C_l(p^0)^{1/2} (x_1-z)\right) dz\\
&=\sup_E\int_E d\sigma_E(p)\ (p^0)^{1/2}\varphi(p)1_{\{p_1>0\}}\frac{1}{\hat{p}_1}\|Q^+(f,f)(\cdot,p)\|_{L^\infty_{x_1}}\\&\qquad\qquad\times \sup_{x_1}\frac{\hat{p}_1}{C_l (p^0)^{1/2}}\left(1-\exp\left(-\frac{1}{\hat{p}_1}C_l(p^0)^{1/2} x_1\right) \right)\\&
\le \sup_E\int_E d\sigma_E(p)\ (p^0)^{1/2}\varphi(p)1_{\{p_1>0\}}\frac{1}{C_l (p^0)^{1/2}}\|Q^+(f,f)(\cdot,p)\|_{L^\infty_{x_1}}\\
&\le \frac{1}{C_l}\sup_E\int_E d\sigma_E(p)\ \varphi(p)\|Q^+(f,f)(\cdot,p)\|_{L^\infty_{x_1}}
\le \frac{1}{C_l} \| f\|^2
\le \frac{a_1^2}{C_l},
\end{align*} by Proposition \ref{prop.3.4} and Proposition \ref{lowerL}. Here, the constants $a_1 $ and $C_l$ are the same as those in the definition of the solution space $\mathcal{A}$ defined in \eqref{sol space}.

Altogether, given $a_1$, $M$ and $C_l$, we obtain $\|f\|_{\textup{hyp}}<+\infty.$ This completes the proof of the hyperplane estimate and hence concludes the
proof of Theorem~\ref{thm:hyperplane}.

\end{proof}

\section*{Acknowledgments}
 J. W. J. is supported by the National Research Foundation of Korea (NRF) grants RS-2023-00210484, RS-2023-00219980, 2022R1G1A1009044 and \\
 2021R1A6A1A10042944. S.-B. Y. is supported by the National Research Foundation of Korea (NRF) grant funded by the Korean government (MSIT) (RS-2023-NR076676).

\bibliographystyle{amsplaindoi}
\bibliography{bibliography.bib}{}

@article {2411.06533,
    AUTHOR = {Wang, Yi and Li, Li and Jiang, Zaihong},
     TITLE = {Existence of solutions to {D}irichlet boundary value problems
              of the stationary relativistic {B}oltzmann equation},
   JOURNAL = {Nonlinearity},
  FJOURNAL = {Nonlinearity},
    VOLUME = {38},
      YEAR = {2025},
    NUMBER = {5},
     PAGES = {Paper No. 055020, 46},
      ISSN = {0951-7715,1361-6544},
   MRCLASS = {35Q35 (35Q20 35Q75 76Y05)},
  MRNUMBER = {4900992},
MRREVIEWER = {Ling\ Xu},
       DOI = {10.1088/1361-6544/adcfe6},
       URL = {https://doi.org/10.1088/1361-6544/adcfe6},
}

@article {MR4891827,
    AUTHOR = {Ouyang, Jing and Xiao, Changguo},
     TITLE = {The steady relativistic {B}oltzmann equation in half-space},
   JOURNAL = {J. Stat. Phys.},
  FJOURNAL = {Journal of Statistical Physics},
    VOLUME = {192},
      YEAR = {2025},
    NUMBER = {4},
     PAGES = {Paper No. 56, 52},
      ISSN = {0022-4715,1572-9613},
   MRCLASS = {82B40 (35B30 35Q20 35Q75)},
  MRNUMBER = {4891827},
       DOI = {10.1007/s10955-025-03421-0},
       URL = {https://doi.org/10.1007/s10955-025-03421-0},
}

@article{2503.12587,
Author = {Ki-Nam Hong and Marwa Shahine and Seok-Bae Yun},
Title = {Stationary {B}oltzmann Equation for Polyatomic Gases in a slab},
Year = {2025},
Journal={arXiv preprint arXiv:2503.12587},
Eprint = {arXiv:2503.12587},
}

@article {MR459450,
    AUTHOR = {Maslova, N. B.},
     TITLE = {Stationary problems for the {B}oltzmann equation in the case
              of large {K}nudsen numbers},
   JOURNAL = {Dokl. Akad. Nauk SSSR},
  FJOURNAL = {Doklady Akademii Nauk SSSR},
    VOLUME = {229},
      YEAR = {1976},
    NUMBER = {3},
     PAGES = {593--596},
      ISSN = {0002-3264},
   MRCLASS = {82.45},
  MRNUMBER = {459450},
MRREVIEWER = {M.\ Dr\u aganu},
}

@article {MR3620020,
    AUTHOR = {Lee, Ho and Nungesser, Ernesto},
     TITLE = {Future global existence and asymptotic behaviour of solutions
              to the {E}instein-{B}oltzmann system with {B}ianchi {I}
              symmetry},
   JOURNAL = {J. Differential Equations},
  FJOURNAL = {Journal of Differential Equations},
    VOLUME = {262},
      YEAR = {2017},
    NUMBER = {11},
     PAGES = {5425--5467},
      ISSN = {0022-0396,1090-2732},
   MRCLASS = {35Q76 (35B40 35Q20 35Q83 35Q85 83C05 83F05)},
  MRNUMBER = {3620020},
MRREVIEWER = {Reinhard\ Redlinger},
       DOI = {10.1016/j.jde.2017.02.004},
       URL = {https://doi.org/10.1016/j.jde.2017.02.004},
}

@article {MR4222138,
    AUTHOR = {Lee, Ho},
     TITLE = {The spatially homogeneous {B}oltzmann equation for massless
              particles in an {FLRW} background},
   JOURNAL = {J. Math. Phys.},
  FJOURNAL = {Journal of Mathematical Physics},
    VOLUME = {62},
      YEAR = {2021},
    NUMBER = {3},
     PAGES = {Paper No. 031502, 15},
      ISSN = {0022-2488,1089-7658},
   MRCLASS = {35Q75 (83C55)},
  MRNUMBER = {4222138},
       DOI = {10.1063/5.0037951},
       URL = {https://doi.org/10.1063/5.0037951},
}

@article {MR4271957,
    AUTHOR = {Jang, Jin Woo and Strain, Robert M. and Yun, Seok-Bae},
     TITLE = {Propagation of uniform upper bounds for the spatially
              homogeneous relativistic {B}oltzmann equation},
   JOURNAL = {Arch. Ration. Mech. Anal.},
  FJOURNAL = {Archive for Rational Mechanics and Analysis},
    VOLUME = {241},
      YEAR = {2021},
    NUMBER = {1},
     PAGES = {149--186},
      ISSN = {0003-9527,1432-0673},
   MRCLASS = {35Q75 (82C40 83A05)},
  MRNUMBER = {4271957},
       DOI = {10.1007/s00205-021-01649-0},
       URL = {https://doi.org/10.1007/s00205-021-01649-0},
}

@article {MR4156121,
    AUTHOR = {Jang, Jin Woo and Yun, Seok-Bae},
     TITLE = {Propagation of {$L^p$} estimates for the spatially homogeneous
              relativistic {B}oltzmann equation},
   JOURNAL = {J. Differential Equations},
  FJOURNAL = {Journal of Differential Equations},
    VOLUME = {272},
      YEAR = {2021},
     PAGES = {105--126},
      ISSN = {0022-0396,1090-2732},
   MRCLASS = {35Q20 (35B65 35Q75 76P05 76Y05 82C40 83A05)},
  MRNUMBER = {4156121},
       DOI = {10.1016/j.jde.2020.09.027},
       URL = {https://doi.org/10.1016/j.jde.2020.09.027},
}

@book{landau1987fluid,
  title={Fluid Mechanics: Volume 6},
  author={Landau, Lev Davidovich and Lifshitz, Evgenii Mikhailovich},
  volume={6},
  year={1987},
  publisher={Elsevier}
}

@article{israel1963relativistic,
  title={Relativistic kinetic theory of a simple gas},
  author={Israel, Werner},
  journal={Journal of Mathematical Physics},
  volume={4},
  number={9},
  pages={1163--1181},
  year={1963},
  publisher={AIP Publishing}
}

@article{israel1979transient,
  title={Transient relativistic thermodynamics and kinetic theory},
  author={Israel, Werner and Stewart, John M},
  journal={Annals of Physics},
  volume={118},
  number={2},
  pages={341--372},
  year={1979},
  publisher={Elsevier}
}

@article{stewart1977transient,
  title={On transient relativistic thermodynamics and kinetic theory},
  author={Stewart, J. M. },
  journal={Proceedings of the Royal Society of London. A. Mathematical and Physical Sciences},
  volume={357},
  number={1688},
  pages={59--75},
  year={1977},
  publisher={The Royal Society London}
}

@misc{synge1958relativistic,
  title={The relativistic gas},
  author={Synge, John Lighton and Morse, Philip M},
  year={1958},
  publisher={American Institute of Physics}
}

@article{van1987chapman,
  title={Chapman-Enskog as an application of the method for eliminating fast variables},
  author={Van Kampen, NG},
  journal={Journal of statistical physics},
  volume={46},
  pages={709--727},
  year={1987},
  publisher={Springer}
}

@article{ruggeri1986relativistic,
  title={Relativistic thermodynamics of gases},
  author={Liu, I. S. and  M{\"u}ller, I. and Ruggeri, T.},
  journal={Ann. of Phys},
  volume={169},
  pages={191},
  year={1986}
}

@article{israel1976thermodynamics,
  title={Thermodynamics of nonstationary and transient effects in a relativistic gas},
  author={Israel, Werner and Stewart, John M},
  journal={Physics Letters A},
  volume={58},
  number={4},
  pages={213--215},
  year={1976},
  publisher={Elsevier}
}

@article{PhysRev.58.269,
  title = {The Thermodynamics of Irreversible Processes. II. Fluid Mixtures},
  author = {Eckart, Carl},
  journal = {Phys. Rev.},
  volume = {58},
  issue = {3},
  pages = {269--275},
  numpages = {0},
  year = {1940},
  month = {Aug},
  publisher = {American Physical Society},
  doi = {10.1103/PhysRev.58.269},
  url = {https://link.aps.org/doi/10.1103/PhysRev.58.269}
}

@misc{2401.00554,
Author = {Hongjie Dong and Yan Guo and Timur Yastrzhembskiy},
Title = {Asymptotic Stability for Relativistic {V}lasov-{M}axwell-{L}andau System in Bounded Domain},
Year = {2023},
Eprint = {arXiv:2401.00554},
}

@article{doi:10.1137/23M1608938,
author = {Dong, Hongjie and Guo, Yan and Ouyang, Zhimeng and Yastrzhembskiy, Timur},
title = {The Local Well-Posedness of the Relativistic {V}lasov–{M}axwell–{L}andau System with the Specular Reflection Boundary Condition},
journal = {SIAM Journal on Mathematical Analysis},
volume = {56},
number = {5},
pages = {6613-6688},
year = {2024},
doi = {10.1137/23M1608938},

URL = { 
    
        https://doi.org/10.1137/23M1608938
    
    

},
eprint = { 
    
        https://doi.org/10.1137/23M1608938
    
    

}
}

@article {MR4470411,
    AUTHOR = {Jang, Jin Woo and Strain, Robert M.},
     TITLE = {Asymptotic stability of the relativistic {B}oltzmann equation
              without angular cut-off},
   JOURNAL = {Ann. PDE},
  FJOURNAL = {Annals of PDE. Journal Dedicated to the Analysis of Problems
              from Physical Sciences},
    VOLUME = {8},
      YEAR = {2022},
    NUMBER = {2},
     PAGES = {Paper No. 20, 167},
      ISSN = {2524-5317,2199-2576},
   MRCLASS = {35Q20 (26A33 35B65 35R11 76P05 82C40)},
  MRNUMBER = {4470411},
MRREVIEWER = {Cecil\ Pompiliu\ Gr\"unfeld},
       DOI = {10.1007/s40818-022-00137-2},
       URL = {https://doi.org/10.1007/s40818-022-00137-2},
}

@article {MR4268835,
    AUTHOR = {Chapman, James and Jang, Jin Woo and Strain, Robert M.},
     TITLE = {On the determinant problem for the relativistic {B}oltzmann
              equation},
   JOURNAL = {Comm. Math. Phys.},
  FJOURNAL = {Communications in Mathematical Physics},
    VOLUME = {384},
      YEAR = {2021},
    NUMBER = {3},
     PAGES = {1913--1943},
      ISSN = {0010-3616,1432-0916},
   MRCLASS = {35Q20 (26B05 76Y05 82C40 83A05)},
  MRNUMBER = {4268835},
       DOI = {10.1007/s00220-021-04101-2},
       URL = {https://doi.org/10.1007/s00220-021-04101-2},
}

@article {MR4264953,
    AUTHOR = {Bae, Gi-Chan and Jang, Jin Woo and Yun, Seok-Bae},
     TITLE = {The relativistic quantum {B}oltzmann equation near
              equilibrium},
   JOURNAL = {Arch. Ration. Mech. Anal.},
  FJOURNAL = {Archive for Rational Mechanics and Analysis},
    VOLUME = {240},
      YEAR = {2021},
    NUMBER = {3},
     PAGES = {1593--1644},
      ISSN = {0003-9527,1432-0673},
   MRCLASS = {35Q40 (82C40)},
  MRNUMBER = {4264953},
       DOI = {10.1007/s00205-021-01643-6},
       URL = {https://doi.org/10.1007/s00205-021-01643-6},
}

@article {MR3926521,
    AUTHOR = {Chen, I-Kun and Hsia, Chun-Hsiung and Kawagoe, Daisuke},
     TITLE = {Regularity for diffuse reflection boundary problem to the
              stationary linearized {B}oltzmann equation in a convex domain},
   JOURNAL = {Ann. Inst. H. Poincar\'e{} C Anal. Non Lin\'eaire},
  FJOURNAL = {Annales de l'Institut Henri Poincar\'e{} C. Analyse Non
              Lin\'eaire},
    VOLUME = {36},
      YEAR = {2019},
    NUMBER = {3},
     PAGES = {745--782},
      ISSN = {0294-1449,1873-1430},
   MRCLASS = {35Q20 (35B65)},
  MRNUMBER = {3926521},
       DOI = {10.1016/j.anihpc.2018.09.002},
       URL = {https://doi.org/10.1016/j.anihpc.2018.09.002},
}

@article {MR406275,
    AUTHOR = {Guiraud, Jean-Pierre},
     TITLE = {Probl\`eme aux limites int\'erieur pour l'\'equation de
              {B}oltzmann en r\'egime stationnaire, faiblement non
              lin\'eaire},
   JOURNAL = {J. M\'ecanique},
  FJOURNAL = {Journal de M\'ecanique},
    VOLUME = {11},
      YEAR = {1972},
     PAGES = {183--231},
      ISSN = {0021-7832},
   MRCLASS = {82.45},
  MRNUMBER = {406275},
MRREVIEWER = {M.\ Dutta},
}

@article {MR4131014,
    AUTHOR = {Esposito, R. and Marra, R.},
     TITLE = {Stationary non equilibrium states in kinetic theory},
   JOURNAL = {J. Stat. Phys.},
  FJOURNAL = {Journal of Statistical Physics},
    VOLUME = {180},
      YEAR = {2020},
    NUMBER = {1-6},
     PAGES = {773--809},
      ISSN = {0022-4715,1572-9613},
   MRCLASS = {82C40 (35Q20 35Q82 76P05)},
  MRNUMBER = {4131014},
       DOI = {10.1007/s10955-020-02528-w},
       URL = {https://doi.org/10.1007/s10955-020-02528-w},
}

@article {MR3085665,
    AUTHOR = {Esposito, Raffaele and Guo, Yan and Kim, Chanwoo and Marra, Rossana},
     TITLE = {Non-isothermal boundary in the {B}oltzmann theory and
              {F}ourier law},
   JOURNAL = {Comm. Math. Phys.},
  FJOURNAL = {Communications in Mathematical Physics},
    VOLUME = {323},
      YEAR = {2013},
    NUMBER = {1},
     PAGES = {177--239},
      ISSN = {0010-3616,1432-0916},
   MRCLASS = {35Q20 (35B30 76P05)},
  MRNUMBER = {3085665},
MRREVIEWER = {Calvin\ Tadmon},
       DOI = {10.1007/s00220-013-1766-2},
       URL = {https://doi.org/10.1007/s00220-013-1766-2},
}

@article {MR4419612,
    AUTHOR = {Chen, Hongxu and Kim, Chanwoo},
     TITLE = {Regularity of stationary {B}oltzmann equation in convex
              domains},
   JOURNAL = {Arch. Ration. Mech. Anal.},
  FJOURNAL = {Archive for Rational Mechanics and Analysis},
    VOLUME = {244},
      YEAR = {2022},
    NUMBER = {3},
     PAGES = {1099--1222},
      ISSN = {0003-9527,1432-0673},
   MRCLASS = {35Q35 (76P05)},
  MRNUMBER = {4419612},
       DOI = {10.1007/s00205-022-01781-5},
       URL = {https://doi.org/10.1007/s00205-022-01781-5},
}

@article {MR1766902,
    AUTHOR = {Arkeryd, Leif and Nouri, Anne},
     TITLE = {On the {M}ilne problem and the hydrodynamic limit for a steady
              {B}oltzmann equation model},
   JOURNAL = {J. Statist. Phys.},
  FJOURNAL = {Journal of Statistical Physics},
    VOLUME = {99},
      YEAR = {2000},
    NUMBER = {3-4},
     PAGES = {993--1019},
      ISSN = {0022-4715,1572-9613},
   MRCLASS = {82C40 (76P05)},
  MRNUMBER = {1766902},
MRREVIEWER = {Carlo\ Cercignani},
       DOI = {10.1023/A:1018655815285},
       URL = {https://doi.org/10.1023/A:1018655815285},
}

@article {MR1413668,
    AUTHOR = {Illner, R. and Struckmeier, J.},
     TITLE = {Boundary value problems for the steady {B}oltzmann equation},
   JOURNAL = {J. Statist. Phys.},
  FJOURNAL = {Journal of Statistical Physics},
    VOLUME = {85},
      YEAR = {1996},
    NUMBER = {3-4},
     PAGES = {427--454},
      ISSN = {0022-4715,1572-9613},
   MRCLASS = {82C40 (76P05)},
  MRNUMBER = {1413668},
MRREVIEWER = {Carlo\ Cercignani},
       DOI = {10.1007/BF02174213},
       URL = {https://doi.org/10.1007/BF02174213},
}

@article {MR2134697,
    AUTHOR = {Arkeryd, Leif and Nouri, Anne},
     TITLE = {A large data existence result for the stationary {B}oltzmann
              equation in a cylindrical geometry},
   JOURNAL = {Ark. Mat.},
  FJOURNAL = {Arkiv f\"or Matematik},
    VOLUME = {43},
      YEAR = {2005},
    NUMBER = {1},
     PAGES = {29--50},
      ISSN = {0004-2080,1871-2487},
   MRCLASS = {35F20 (82C40)},
  MRNUMBER = {2134697},
MRREVIEWER = {Manuel\ Portilheiro},
       DOI = {10.1007/BF02383609},
       URL = {https://doi.org/10.1007/BF02383609},
}

@article {MR2130642,
    AUTHOR = {Arkeryd, Leif and Nouri, Anne},
     TITLE = {The stationary nonlinear {B}oltzmann equation in a {C}ouette
              setting with multiple, isolated {$L^q$}-solutions and
              hydrodynamic limits},
   JOURNAL = {J. Stat. Phys.},
  FJOURNAL = {Journal of Statistical Physics},
    VOLUME = {118},
      YEAR = {2005},
    NUMBER = {5-6},
     PAGES = {849--881},
      ISSN = {0022-4715,1572-9613},
   MRCLASS = {82C40 (35B45 35F20 76P05)},
  MRNUMBER = {2130642},
MRREVIEWER = {Carlo\ Cercignani},
       DOI = {10.1007/s10955-004-2708-3},
       URL = {https://doi.org/10.1007/s10955-004-2708-3},
}

@article {MR1991144,
    AUTHOR = {Arkeryd, Leif and Nouri, Anne},
     TITLE = {The stationary {B}oltzmann equation in {$\Bbb R^n$} with given
              indata},
   JOURNAL = {Ann. Sc. Norm. Super. Pisa Cl. Sci. (5)},
  FJOURNAL = {Annali della Scuola Normale Superiore di Pisa. Classe di
              Scienze. Serie V},
    VOLUME = {1},
      YEAR = {2002},
    NUMBER = {2},
     PAGES = {359--385},
      ISSN = {0391-173X,2036-2145},
   MRCLASS = {35F30 (35Q35 76P05 82C40)},
  MRNUMBER = {1991144},
MRREVIEWER = {C\'edric\ Villani},
}

@article {MR1842024,
    AUTHOR = {Arkeryd, Leif and Nouri, Anne},
     TITLE = {{$L^1$} solutions to the stationary {B}oltzmann equation in a
              slab},
   JOURNAL = {Ann. Fac. Sci. Toulouse Math. (6)},
  FJOURNAL = {Annales de la Facult\'e{} des Sciences de Toulouse.
              Math\'ematiques. S\'erie 6},
    VOLUME = {9},
      YEAR = {2000},
    NUMBER = {3},
     PAGES = {375--413},
      ISSN = {0240-2963,2258-7519},
   MRCLASS = {82C40},
  MRNUMBER = {1842024},
MRREVIEWER = {Carlo\ Cercignani},
       URL = {http://www.numdam.org/item?id=AFST_2000_6_9_3_375_0},
}

@article {MR1375351,
    AUTHOR = {Arkeryd, Leif and Nouri, Anne},
     TITLE = {A compactness result related to the stationary {B}oltzmann
              equation in a slab, with applications to the existence theory},
   JOURNAL = {Indiana Univ. Math. J.},
  FJOURNAL = {Indiana University Mathematics Journal},
    VOLUME = {44},
      YEAR = {1995},
    NUMBER = {3},
     PAGES = {815--839},
      ISSN = {0022-2518,1943-5258},
   MRCLASS = {82C40 (35Q99 76P05)},
  MRNUMBER = {1375351},
MRREVIEWER = {Jian\ Ning\ Chen},
       DOI = {10.1512/iumj.1995.44.2010},
       URL = {https://doi.org/10.1512/iumj.1995.44.2010},
}

@book {MR1744523,
    AUTHOR = {Cercignani, Carlo},
     TITLE = {Rarefied gas dynamics},
    SERIES = {Cambridge Texts in Applied Mathematics},
      NOTE = {From basic concepts to actual calculations},
 PUBLISHER = {Cambridge University Press, Cambridge},
      YEAR = {2000},
     PAGES = {xviii+320},
      ISBN = {0-521-65008-9; 0-521-65992-2},
   MRCLASS = {76P05 (76M25 76V05 82-01 82C40)},
  MRNUMBER = {1744523},
MRREVIEWER = {Hans\ Babovsky},
}

@article{Juttner,
author = {Jüttner, Ferencz},
title = {Das Maxwellsche Gesetz der Geschwindigkeitsverteilung in der Relativtheorie},
journal = {Annalen der Physik},
volume = {339},
number = {5},
pages = {856-882},
doi = {https://doi.org/10.1002/andp.19113390503},
url = {https://onlinelibrary.wiley.com/doi/abs/10.1002/andp.19113390503},
eprint = {https://onlinelibrary.wiley.com/doi/pdf/10.1002/andp.19113390503},
year = {1911}
}

@article {MR122493,
    AUTHOR = {Abonyi, Iv\`an},
     TITLE = {Steady state solution of the relativistic {B}oltzmann
              transport equation},
   JOURNAL = {Z. Angew. Math. Phys.},
  FJOURNAL = {Zeitschrift f\"ur Angewandte Mathematik und Physik. ZAMP.
              Journal of Applied Mathematics and Physics. Journal de
              Math\'ematiques et de Physique Appliqu\'ees},
    VOLUME = {11},
      YEAR = {1960},
     PAGES = {169--175},
      ISSN = {0044-2275,1420-9039},
   MRCLASS = {83.00 (82.00)},
  MRNUMBER = {122493},
       DOI = {10.1007/BF02155708},
       URL = {https://doi.org/10.1007/BF02155708},
}

@article {MR2122556,
    AUTHOR = {Ghomeshi, S.},
     TITLE = {Existence and uniqueness of solutions for the {C}ouette
              problem},
   JOURNAL = {J. Stat. Phys.},
  FJOURNAL = {Journal of Statistical Physics},
    VOLUME = {118},
      YEAR = {2005},
    NUMBER = {1-2},
     PAGES = {265--300},
      ISSN = {0022-4715,1572-9613},
   MRCLASS = {35F20 (76P05 82C40)},
  MRNUMBER = {2122556},
MRREVIEWER = {Changjiang\ Zhu},
       DOI = {10.1007/s10955-004-8784-6},
       URL = {https://doi.org/10.1007/s10955-004-8784-6},
}

@article{MR2982812,
	Author = {Ha, Seung-Yeal and Jeong, Eunhee and Strain, Robert M.},
	Doi = {10.3934/cpaa.2013.12.1141},
	Fjournal = {Communications on Pure and Applied Analysis},
	Issn = {1534-0392},
	Journal = {Commun. Pure Appl. Anal.},
	Mrclass = {35Q20 (35B35 35Q75 82C40)},
	Mrnumber = {2982812},
	Mrreviewer = {Calvin Tadmon},
	Number = {2},
	Pages = {1141--1161},
	Title = {Uniform {$L^1$}-stability of the relativistic {B}oltzmann equation near vacuum},
	Url = {https://doi.org/10.3934/cpaa.2013.12.1141},
	Volume = {12},
	Year = {2013}}

@article{landau,
	Author = {Landau, L.D.},
	Journal = {J. Phys. USSR},
	Note = {English translation in {\em JETP 16}, 574. Reproduced in {\em Collected papers of L.D. Landau}, edited and with an introduction by D. ter Haar, Pergamon Press, 1965, pp.~445--460; and in {\em Men of Physics: L.D. Landau}, Vol. 2, Pergamon Press, D. ter Haar, ed. (1965). \footnote{There is a misprint in formula (17) of this reference (p.~104): replace $e^{-(ka)^2/2}$ by $e^{-1/(2(ka)^2)}$.}},
	Pages = {25},
	Title = {On the vibration of the electronic plasma},
	Volume = {10},
	Year = {1946}}

@article{1903.05301,
	Author = {Robert M. Strain and Zhenfu Wang},
	Eprint = {arXiv:1903.05301},
	Fjournal = {Quarterly of Applied Mathematics},
	Journal = {Quart. Appl. Math.},
	Pages = {1--39},
	Title = {Uniqueness of Bounded Solutions for the Homogeneous Relativistic {L}andau Equation with {C}oulomb Interactions},
	Year = {2019}}

@article{CM,
	Author = {Carrapatoso, K. and Mischler, S.},
	Doi = {10.1007/s40818-017-0021-0},
	Fjournal = {Annals of PDE. Journal Dedicated to the Analysis of Problems from Physical Sciences},
	Issn = {2199-2576},
	Journal = {Ann. PDE},
	Mrclass = {35F25 (35B35 35B40 35Q20 47D06 47H20)},
	Mrnumber = {3625186},
	Mrreviewer = {Hidetoshi Tahara},
	Number = {1},
	Pages = {Art. 1, 65},
	Title = {{L}andau equation for very soft and {C}oulomb potentials near {M}axwellians},
	Url = {https://doi.org/10.1007/s40818-017-0021-0},
	Volume = {3},
	Year = {2017}}

@article{EGKM-18-AP,
	Author = {Esposito, Raffaele and Guo, Yan and Kim, Chanwoo and Marra, Rossana},
	Doi = {10.1007/s40818-017-0037-5},
	Fjournal = {Annals of PDE. Journal Dedicated to the Analysis of Problems from Physical Sciences},
	Issn = {2199-2576},
	Journal = {Ann. PDE},
	Mrclass = {35Q20 (35Q30 76P05)},
	Mrnumber = {3740632},
	Mrreviewer = {Bruno Scheurer},
	Number = {1},
	Pages = {Art. 1, 119},
	Title = {Stationary solutions to the {B}oltzmann equation in the hydrodynamic limit},
	Url = {https://doi.org/10.1007/s40818-017-0037-5},
	Volume = {4},
	Year = {2018}}

@incollection{Vil02,
	Address = {Amsterdam},
	Author = {Villani, C{\'{e}}dric},
	Booktitle = {Handbook of mathematical fluid dynamics, {V}ol. {I}},
	Doi = {10.1016/S1874-5792(02)80004-0},
	Mrclass = {82C40 (35F20 76P05 82-02)},
	Mrnumber = {1942465},
	Mrreviewer = {Fran\c{c}ois Castella},
	Pages = {71--305},
	Publisher = {North-Holland},
	Title = {A review of mathematical topics in collisional kinetic theory},
	Url = {https://doi.org/10.1016/S1874-5792(02)80004-0},
	Year = {2002}}

@article{MR1211782,
	Author = {Glassey, Robert T. and Strauss, Walter A.},
	Doi = {10.2977/prims/1195167275},
	Fjournal = {Kyoto University. Research Institute for Mathematical Sciences. Publications},
	Issn = {0034-5318},
	Journal = {Publ. Res. Inst. Math. Sci.},
	Mrclass = {82C40 (76P05)},
	Mrnumber = {1211782},
	Mrreviewer = {Carlo Cercignani},
	Number = {2},
	Pages = {301--347},
	Title = {Asymptotic stability of the relativistic {M}axwellian},
	Url = {https://doi.org/10.2977/prims/1195167275},
	Volume = {29},
	Year = {1993}}

@book{MR1898707,
	Author = {Cercignani, Carlo and Kremer, Gilberto Medeiros},
	Doi = {10.1007/978-3-0348-8165-4},
	Isbn = {3-7643-6693-1},
	Mrclass = {82C40 (76P05 76V05 83C55)},
	Mrnumber = {1898707},
	Mrreviewer = {John M. Stewart},
	Pages = {x+384},
	Publisher = {Birkh\"auser Verlag}, Address ={Basel}, 
 Location = {Switzerland},
	Series = {Progress in Mathematical Physics},
	Title = {The relativistic {B}oltzmann equation: theory and applications},
	Url = {https://doi.org/10.1007/978-3-0348-8165-4},
	Volume = {22},
	Year = {2002}}

@article{Dudynski2,
	Author = {Dudy\'{n}ski, Marek and Ekiel-Je\.{z}ewska, Maria L.},
	Journal = {J. Tech. Phys.},
	Pages = {39-47},
	Title = {The relativistic {B}oltzmann equation - mathematical and physical aspects},
	Volume = {48},
	Year = {2007}}

@article{MR2543323,
	Author = {Ha, Seung-Yeal and Lee, Ho and Yang, Xiongfeng and Yun, Seok-Bae},
	Doi = {10.1142/S0219891609001848},
	Fjournal = {Journal of Hyperbolic Differential Equations},
	Issn = {0219-8916},
	Journal = {J. Hyperbolic Differ. Equ.},
	Mrclass = {35Q20 (35Q75 76P05 82C40)},
	Mrnumber = {2543323},
	Number = {2},
	Pages = {295--312},
	Title = {Uniform {$L^2$}-stability estimates for the relativistic {B}oltzmann equation},
	Url = {https://doi.org/10.1142/S0219891609001848},
	Volume = {6},
	Year = {2009}}

@article{MR3189734,
	Author = {Kremer, Gilberto M.},
	Doi = {10.1142/S0219887814600056},
	Fjournal = {International Journal of Geometric Methods in Modern Physics},
	Issn = {0219-8878},
	Journal = {Int. J. Geom. Methods Mod. Phys.},
	Mrclass = {76P05 (35Q20 35Q75 83C10)},
	Mrnumber = {3189734},
	Mrreviewer = {Silvia Caprino},
	Number = {2},
	Pages = {1460005, 16},
	Title = {Theory and applications of the relativistic {B}oltzmann equation},
	Url = {https://doi.org/10.1142/S0219887814600056},
	Volume = {11},
	Year = {2014}}

@article{MR2793935,
	Author = {Speck, Jared and Strain, Robert M.},
	Doi = {10.1007/s00220-011-1207-z},
	Fjournal = {Communications in Mathematical Physics},
	Issn = {0010-3616},
	Journal = {Comm. Math. Phys.},
	Mrclass = {82C40 (35Q20 76P05 76Y05)},
	Mrnumber = {2793935},
	Mrreviewer = {Stephen Wollman},
	Number = {1},
	Pages = {229--280},
	Title = {Hilbert expansion from the {B}oltzmann equation to relativistic fluids},
	Url = {https://doi.org/10.1007/s00220-011-1207-z},
	Volume = {304},
	Year = {2011}}

@book{Stewart,
	Author = {Stewart, John M.},
	Doi = {10.1007/BFb0025374},
	Isbn = {978-3-540-36940-0},
	Pages = {III, 117},
	Publisher = {Springer, Berlin, Heidelberg},
	Title = {Non-Equilibrium Relativistic Kinetic Theory},
	Url = {https://doi.org/10.1007/BFb0025374},
	Year = {1971}}

@phdthesis{MR2707256,
	Author = {Strain, Robert Mills},
	Isbn = {978-0542-12875-2},
	Journal = {ProQuest Dissertations and Theses},
	Language = {English},
	Mrclass = {Thesis},
	Mrnumber = {2707256},
	Note = {(ProQuest Document ID 305028444)},
	Pages = {1--200},
	Publisher = {ProQuest LLC, Ann Arbor, MI},
	School = {Brown University},
	Title = {Some applications of an energy method in collisional kinetic theory},
	Url = {http://gateway.proquest.com/openurl?url_ver=Z39.88-2004&rft_val_fmt=info:ofi/fmt:kev:mtx:dissertation&res_dat=xri:pqdiss&rft_dat=xri:pqdiss:3174679},
	Year = {2005}}

@article{MR2728733,
	Author = {Strain, Robert M.},
	Doi = {10.1007/s00220-010-1129-1},
	Fjournal = {Communications in Mathematical Physics},
	Issn = {0010-3616},
	Journal = {Comm. Math. Phys.},
	Mrclass = {82C40 (35Q20)},
	Mrnumber = {2728733},
	Mrreviewer = {Stephen Wollman},
	Number = {2},
	Pages = {529--597},
	Title = {Asymptotic stability of the relativistic {B}oltzmann equation for the soft potentials},
	Url = {https://doi.org/10.1007/s00220-010-1129-1},
	Volume = {300},
	Year = {2010}}

@article{MR2679588,
	Author = {Strain, Robert M.},
	Doi = {10.1137/090762695},
	Fjournal = {SIAM Journal on Mathematical Analysis},
	Issn = {0036-1410},
	Journal = {SIAM J. Math. Anal.},
	Mrclass = {82C40 (35Q20 35Q75 76P05)},
	Mrnumber = {2679588},
	Mrreviewer = {Silvia Lorenzani},
	Number = {4},
	Pages = {1568--1601},
	Title = {Global {N}ewtonian limit for the relativistic {B}oltzmann equation near vacuum},
	Url = {https://doi.org/10.1137/090762695},
	Volume = {42},
	Year = {2010}}

@article{MR2765751,
	Author = {Strain, Robert M.},
	Doi = {10.3934/krm.2011.4.345},
	Fjournal = {Kinetic and Related Models},
	Issn = {1937-5093},
	Journal = {Kinet. Relat. Models},
	Mrclass = {82C40 (76P05 83A05)},
	Mrnumber = {2765751},
	Mrreviewer = {Stephen Wollman},
	Number = {1},
	Pages = {345--359},
	Title = {Coordinates in the relativistic {B}oltzmann theory},
	Url = {https://doi.org/10.3934/krm.2011.4.345},
	Volume = {4},
	Year = {2011}}

@book{C-I-P,
	Author = {Cercignani, Carlo and Illner, Reinhard and Pulvirenti, Mario},
	Doi = {10.1007/978-1-4419-8524-8},
	Isbn = {0-387-94294-7},
	Mrclass = {82C40 (76-02 76P05 82-02 82B40)},
	Mrnumber = {1307620},
	Mrreviewer = {Giuseppe Toscani},
	Pages = {viii+347},
	Publisher = {Springer-Verlag, New York},
	Series = {Applied Mathematical Sciences},
	Title = {The mathematical theory of dilute gases},
	Url = {https://doi.org/10.1007/978-1-4419-8524-8},
	Volume = {106},
	Year = {1994}}

@book{C,
	Author = {Cercignani, Carlo},
	Doi = {10.1007/978-1-4612-1039-9},
	Isbn = {0-387-96637-4},
	Mrclass = {82C40 (35Q99 45G10 76P05)},
	Mrnumber = {1313028},
	Pages = {xii+455},
	Publisher = {Springer-Verlag, New York},
	Series = {Applied Mathematical Sciences},
	Title = {The {B}oltzmann equation and its applications},
	Url = {https://doi.org/10.1007/978-1-4612-1039-9},
	Volume = {67},
	Year = {1988}}

@book{DeGroot,
	Address = {Amsterdam-New York},
	Author = {de Groot, S. R. and van Leeuwen, W. A. and van Weert, Ch. G.},
	Isbn = {0-444-85453-3},
	Mrclass = {82A70 (76Y05 82-02 83-02 85-02)},
	Mrnumber = {635279},
	Mrreviewer = {Carlo Cercignani},
	Pages = {xvii+417},
	Publisher = {North-Holland Publishing Co.},
	Title = {Relativistic {K}inetic {T}heory. Principles and applications.},
	Year = {1980}}

@article{D-E0er,
	Author = {Dudy\'{n}ski, Marek and Ekiel-Je\.{z}ewska, Maria L.},
	Fjournal = {Investigaci\'{o}n Operacional},
	Issn = {0257-4306},
	Journal = {Investigaci\'{o}n Oper.},
	Mrclass = {82A40 (76Y05)},
	Mrnumber = {841735},
	Number = {1},
	Pages = {2228},
	Title = {Errata: ``{C}ausality of the linearized relativistic {B}oltzmann equation''},
	Volume = {6},
	Year = {1985}}

@article{D-E0,
	Author = {Dudy\'{n}ski, Marek and Ekiel-Je\.{z}ewska, Maria L.},
	Doi = {10.1103/PhysRevLett.55.2831},
	Fjournal = {Physical Review Letters},
	Issn = {0031-9007},
	Journal = {Phys. Rev. Lett.},
	Mrclass = {82A40 (76Y05)},
	Mrnumber = {818441},
	Number = {26},
	Pages = {2831--2834},
	Title = {Causality of the linearized relativistic {B}oltzmann equation},
	Url = {https://doi.org/10.1103/PhysRevLett.55.2831},
	Volume = {55},
	Year = {1985}}

@article{D-E2,
	Author = {Dudy\'{n}ski, Marek and Ekiel-Je\.{z}ewska, Maria L.},
	Doi = {10.1007/BF01055712},
	Fjournal = {Journal of Statistical Physics},
	Issn = {0022-4715},
	Journal = {J. Statist. Phys.},
	Mrclass = {82C40 (76P05)},
	Mrnumber = {1151987},
	Mrreviewer = {Luisa Arlotti},
	Number = {3-4},
	Pages = {991--1001},
	Title = {Global existence proof for relativistic {B}oltzmann equation},
	Url = {https://doi.org/10.1007/BF01055712},
	Volume = {66},
	Year = {1992}}

@article{D,
	Author = {Dudy\'{n}ski, Marek},
	Doi = {10.1007/BF01023641},
	Fjournal = {Journal of Statistical Physics},
	Issn = {0022-4715},
	Journal = {J. Statist. Phys.},
	Mrclass = {82C40 (76P05 76Y05)},
	Mrnumber = {1031410},
	Mrreviewer = {Carlo Cercignani},
	Number = {1-2},
	Pages = {199--245},
	Title = {On the linearized relativistic {B}oltzmann equation. {II}. {E}xistence of hydrodynamics},
	Url = {https://doi.org/10.1007/BF01023641},
	Volume = {57},
	Year = {1989}}

@article{D-E3,
	Author = {Dudy\'{n}ski, Marek and Ekiel-Je\.{z}ewska, Maria L.},
	Fjournal = {Communications in Mathematical Physics},
	Issn = {0010-3616},
	Journal = {Comm. Math. Phys.},
	Mrclass = {82A40 (76P05 76Y05)},
	Mrnumber = {933458},
	Mrreviewer = {Carlo Cercignani},
	Number = {4},
	Pages = {607--629},
	Title = {On the linearized relativistic {B}oltzmann equation. {I}. {E}xistence of solutions},
	Url = {http://projecteuclid.org/euclid.cmp/1104161087},
	Volume = {115},
	Year = {1988}}

@book{E-M-V,
	Author = {Escobedo, Miguel and Mischler, St\'{e}phane and Valle, Manuel A.},
	Mrclass = {82C40},
	Mrnumber = {1958975},
	Mrreviewer = {Carlo Cercignani},
	Pages = {85},
	Publisher = {Southwest Texas State University, San Marcos, TX},
	Series = {Electronic Journal of Differential Equations. Monograph},
	Title = {Homogeneous {B}oltzmann equation in quantum relativistic kinetic theory},
	Url = {https://ejde.math.txstate.edu/Monographs/04/abstr.html},
	Volume = {4},
	Year = {2003}}

@article{GL-Vacuum,
	Author = {Glassey, Robert T.},
	Doi = {10.1007/s00220-006-1522-y},
	Fjournal = {Communications in Mathematical Physics},
	Issn = {0010-3616},
	Journal = {Comm. Math. Phys.},
	Mrclass = {82C40 (35F25 35Q75 76P05 76Y05)},
	Mrnumber = {2217287},
	Mrreviewer = {Carlo Cercignani},
	Number = {3},
	Pages = {705--724},
	Title = {Global solutions to the {C}auchy problem for the relativistic {B}oltzmann equation with near-vacuum data},
	Url = {https://doi.org/10.1007/s00220-006-1522-y},
	Volume = {264},
	Year = {2006}}

@book{GL1996,
	Author = {Glassey, Robert T.},
	Doi = {10.1137/1.9781611971477},
	Isbn = {0-89871-367-6},
	Mrclass = {82C40 (35Q99 76P05 76X05 82-02 82D10)},
	Mrnumber = {1379589},
	Mrreviewer = {Andrzej Palczewski},
	Pages = {xii+241},
	Publisher = {Society for Industrial and Applied Mathematics (SIAM), Philadelphia, PA},
	Title = {The {C}auchy problem in kinetic theory},
	Url = {https://doi.org/10.1137/1.9781611971477},
	Year = {1996}}

@article{GS4,
	Author = {Glassey, Robert T. and Strauss, Walter A.},
	Doi = {10.1080/00411459508206020},
	Fjournal = {Transport Theory and Statistical Physics},
	Issn = {0041-1450},
	Journal = {Transport Theory Statist. Phys.},
	Mrclass = {82C40 (76P05)},
	Mrnumber = {1321370},
	Mrreviewer = {Carlo Cercignani},
	Number = {4-5},
	Pages = {657--678},
	Title = {Asymptotic stability of the relativistic {M}axwellian via fourteen moments},
	Url = {https://doi.org/10.1080/00411459508206020},
	Volume = {24},
	Year = {1995}}

@article{Guo-Strain2,
	Author = {Guo, Yan and Strain, Robert M.},
	Doi = {10.1007/s00220-012-1417-z},
	Fjournal = {Communications in Mathematical Physics},
	Issn = {0010-3616},
	Journal = {Comm. Math. Phys.},
	Mrclass = {82D10 (35B35 35B65 35Q83 82C24 82C40)},
	Mrnumber = {2891870},
	Mrreviewer = {Stephen Wollman},
	Number = {3},
	Pages = {649--673},
	Title = {Momentum regularity and stability of the relativistic {V}lasov-{M}axwell-{B}oltzmann system},
	Url = {https://doi.org/10.1007/s00220-012-1417-z},
	Volume = {310},
	Year = {2012}}

@article{Guo-Strain3,
	Author = {Strain, Robert M. and Guo, Yan},
	Doi = {10.1007/s00220-004-1151-2},
	Fjournal = {Communications in Mathematical Physics},
	Issn = {0010-3616},
	Journal = {Comm. Math. Phys.},
	Mrclass = {82D10 (35Q60 35Q75 76X05 76Y05)},
	Mrnumber = {2100057},
	Mrreviewer = {C\'{e}dric Villani},
	Number = {2},
	Pages = {263--320},
	Title = {Stability of the relativistic {M}axwellian in a collisional plasma},
	Url = {https://doi.org/10.1007/s00220-004-1151-2},
	Volume = {251},
	Year = {2004}}

@article{H-R-Yun,
author = {Hwang, Byung-Hoon and Ruggeri, Tommaso and Yun, Seok-Bae},
title = {On a Relativistic {BGK} Model for Polyatomic Gases Near Equilibrium},
journal = {SIAM Journal on Mathematical Analysis},
volume = {54},
number = {3},
pages = {2906-2947},
year = {2022},
doi = {10.1137/21M1404946},

URL = { 
    
        https://doi.org/10.1137/21M1404946
    
    

},
eprint = { 
    
        https://doi.org/10.1137/21M1404946
    
    

}
}

@article {1811.10023,
    AUTHOR = {Hwang, Byung-Hoon and Yun, Seok-Bae},
     TITLE = {Anderson-{W}itting model of the relativistic {B}oltzmann
              equation near equilibrium},
   JOURNAL = {J. Stat. Phys.},
  FJOURNAL = {Journal of Statistical Physics},
    VOLUME = {176},
      YEAR = {2019},
    NUMBER = {4},
     PAGES = {1009--1045},
      ISSN = {0022-4715,1572-9613},
   MRCLASS = {35Q75 (35Q20)},
  MRNUMBER = {3990222},
       DOI = {10.1007/s10955-019-02330-3},
       URL = {https://doi.org/10.1007/s10955-019-02330-3},
}

@article {1801.08382,
    AUTHOR = {Hwang, Byung-Hoon and Yun, Seok-Bae},
     TITLE = {Stationary solutions to the boundary value problem for the
              relativistic {BGK} model in a slab},
   JOURNAL = {Kinet. Relat. Models},
  FJOURNAL = {Kinetic and Related Models},
    VOLUME = {12},
      YEAR = {2019},
    NUMBER = {4},
     PAGES = {749--764},
      ISSN = {1937-5093,1937-5077},
   MRCLASS = {35Q75 (35Q20 76Y05 83A05)},
  MRNUMBER = {3984749},
       DOI = {10.3934/krm.2019029},
       URL = {https://doi.org/10.3934/krm.2019029},
}

@article{MR3880739,
	Author = {Jang, Jin Woo and Yun, Seok-Bae},
	Doi = {10.1016/j.aml.2018.11.001},
	Fjournal = {Applied Mathematics Letters. An International Journal of Rapid Publication},
	Issn = {0893-9659},
	Journal = {Appl. Math. Lett.},
	Mrclass = {83A05 (82C40)},
	Mrnumber = {3880739},
	Pages = {162--169},
	Title = {Gain of regularity for the relativistic collision operator},
	Url = {https://doi.org/10.1016/j.aml.2018.11.001},
	Volume = {90},
	Year = {2019}}

@article{MR2378164,
	Author = {Jiang, Zhenglu},
	Doi = {10.1007/s10955-007-9453-3},
	Fjournal = {Journal of Statistical Physics},
	Issn = {0022-4715},
	Journal = {J. Stat. Phys.},
	Mrclass = {35F20 (35Q75 76P05 82C40)},
	Mrnumber = {2378164},
	Number = {3},
	Pages = {535--544},
	Title = {Global existence proof for relativistic {B}oltzmann equation with hard interactions},
	Url = {https://doi.org/10.1007/s10955-007-9453-3},
	Volume = {130},
	Year = {2008}}

@article{MR3169776,
	Author = {Lee, Ho and Rendall, Alan D.},
	Doi = {10.1080/03605302.2013.827709},
	Fjournal = {Communications in Partial Differential Equations},
	Issn = {0360-5302},
	Journal = {Comm. Partial Differential Equations},
	Mrclass = {35Q75 (35Q20 82C40 83C55)},
	Mrnumber = {3169776},
	Mrreviewer = {Calvin Tadmon},
	Number = {12},
	Pages = {2238--2262},
	Title = {The spatially homogeneous relativistic {B}oltzmann equation with a hard potential},
	Url = {https://doi.org/10.1080/03605302.2013.827709},
	Volume = {38},
	Year = {2013}}

@article{Cal,
	Author = {Calogero, Simone},
	Doi = {10.1063/1.1793328},
	Fjournal = {Journal of Mathematical Physics},
	Issn = {0022-2488},
	Journal = {J. Math. Phys.},
	Mrclass = {82C40},
	Mrnumber = {2098116},
	Mrreviewer = {Carlo Cercignani},
	Number = {11},
	Pages = {4042--4052},
	Title = {The {N}ewtonian limit of the relativistic {B}oltzmann equation},
	Url = {https://doi.org/10.1063/1.1793328},
	Volume = {45},
	Year = {2004}}

@article{MR2988960,
	Author = {Bellouquid, Abdelghani and Calvo, Juan and Nieto, Juanjo and Soler, Juan},
	Doi = {10.1007/s10955-012-0600-0},
	Fjournal = {Journal of Statistical Physics},
	Issn = {0022-4715},
	Journal = {J. Stat. Phys.},
	Mrclass = {82Cxx},
	Mrnumber = {2988960},
	Number = {2},
	Pages = {284--316},
	Title = {On the relativistic {BGK}-{B}oltzmann model: asymptotics and hydrodynamics},
	Url = {https://doi.org/10.1007/s10955-012-0600-0},
	Volume = {149},
	Year = {2012}}

@article{MR3300786,
	Author = {Bellouquid, A. and Nieto, J. and Urrutia, L.},
	Doi = {10.1016/j.na.2014.10.020},
	Fjournal = {Nonlinear Analysis. Theory, Methods \& Applications. An International Multidisciplinary Journal},
	Issn = {0362-546X},
	Journal = {Nonlinear Anal.},
	Mrclass = {35Q75 (35A01 35B35 35Q20)},
	Mrnumber = {3300786},
	Pages = {87--104},
	Title = {Global existence and asymptotic stability near equilibrium for the relativistic {BGK} model},
	Url = {https://doi.org/10.1016/j.na.2014.10.020},
	Volume = {114},
	Year = {2015}}

@article{MR2102321,
	Author = {Andr\'{e}asson, H{\aa}kan and Calogero, Simone and Illner, Reinhard},
	Doi = {10.1002/mma.555},
	Fjournal = {Mathematical Methods in the Applied Sciences},
	Issn = {0170-4214},
	Journal = {Math. Methods Appl. Sci.},
	Mrclass = {82C40 (35B40 35F25 35Q75 76P05 76Y05)},
	Mrnumber = {2102321},
	Mrreviewer = {Song Mu Zheng},
	Number = {18},
	Pages = {2231--2240},
	Title = {On blowup for gain-term-only classical and relativistic {B}oltzmann equations},
	Url = {https://doi.org/10.1002/mma.555},
	Volume = {27},
	Year = {2004}}

@article{MR1402446,
	Author = {Andr\'{e}asson, H{\aa}kan},
	Doi = {10.1137/0527076},
	Fjournal = {SIAM Journal on Mathematical Analysis},
	Issn = {0036-1410},
	Journal = {SIAM J. Math. Anal.},
	Mrclass = {76P05 (82C40)},
	Mrnumber = {1402446},
	Mrreviewer = {Gheorghe Procopiuc},
	Number = {5},
	Pages = {1386--1405},
	Title = {Regularity of the gain term and strong {$L^1$} convergence to equilibrium for the relativistic {B}oltzmann equation},
	Url = {https://doi.org/10.1137/0527076},
	Volume = {27},
	Year = {1996}}

@article{Pennisi_2018,
	Author = {Sebastiano Pennisi and Tommaso Ruggeri},
	Doi = {10.1088/1742-6596/1035/1/012005},
	Journal = {Journal of Physics: Conference Series},
	Pages = {012005},
	Publisher = {{IOP} Publishing},
	Title = {A New {BGK} Model for Relativistic Kinetic Theory of Monatomic and Polyatomic Gases},
	Url = {https://doi.org/10.1088%2F1742-6596%2F1035%2F1%2F012005},
	Volume = {1035},
	Year = 2018}

@article{MR3166961,
	Author = {Strain, Robert M. and Yun, Seok-Bae},
	Doi = {10.1137/130923531},
	Fjournal = {SIAM Journal on Mathematical Analysis},
	Issn = {0036-1410},
	Journal = {SIAM J. Math. Anal.},
	Mrclass = {35Q20 (76P05 76Y05)},
	Mrnumber = {3166961},
	Mrreviewer = {Piotr Biler},
	Number = {1},
	Pages = {917--938},
	Title = {Spatially homogeneous {B}oltzmann equation for relativistic particles},
	Url = {https://doi.org/10.1137/130923531},
	Volume = {46},
	Year = {2014}}

@article{StrainTas,
	Author = {Robert M. Strain and Maja Taskovi{\'{c}}},
	Doi = {10.1016/j.jfa.2019.04.007},
	Fjournal = {Journal of Functional Analysis},
	Journal = {J. Funct. Anal.},
	Pages = {1--50},
	Title = {{E}ntropy dissipation estimates for the relativistic {L}andau equation, and applications},
	Url = {https://doi.org/10.1016/j.jfa.2019.04.007},
	Year = {2019}}

\end{document}